\documentclass[a4paper,11pt,reqno]{amsart}

\usepackage[utf8]{inputenc}
\usepackage[T1]{fontenc}
\usepackage{lmodern}
\usepackage[english]{babel}
\usepackage{microtype}

\usepackage{amsmath,amssymb,amsfonts,amsthm}
\usepackage{lmodern}
\usepackage{mathtools,accents}
\usepackage{mathrsfs}
\usepackage{aliascnt}
\usepackage{braket}
\usepackage{bm}

\usepackage[a4paper,margin=3cm]{geometry}
\usepackage[citecolor=blue,colorlinks]{hyperref}
\usepackage[english,capitalize]{cleveref}

\usepackage{enumerate}
\usepackage{xcolor}

\makeatletter
\g@addto@macro\@floatboxreset\centering
\makeatother

\makeatletter
\def\newaliasedtheorem#1[#2]#3{
	\newaliascnt{#1@alt}{#2}
	\newtheorem{#1}[#1@alt]{#3}
	\expandafter\newcommand\csname #1@altname\endcsname{#3}
}
\makeatother

\numberwithin{equation}{section}

\newtheoremstyle{slanted}{\topsep}{\topsep}{\slshape}{}{\bfseries}{.}{.5em}{}

\theoremstyle{plain}
\newtheorem{theorem}{Theorem}[section]
\newaliasedtheorem{proposition}[theorem]{Proposition}
\newaliasedtheorem{lemma}[theorem]{Lemma}
\newaliasedtheorem{corollary}[theorem]{Corollary}
\newaliasedtheorem{counterexample}[theorem]{Counterexample}

\theoremstyle{definition}
\newaliasedtheorem{definition}[theorem]{Definition}
\newaliasedtheorem{question}[theorem]{Question}
\newaliasedtheorem{conjecture}[theorem]{Conjecture}

\theoremstyle{remark}
\newaliasedtheorem{remark}[theorem]{Remark}
\newaliasedtheorem{example}[theorem]{Example}

\newcommand{\eps}{\varepsilon}

\let\altphi\phi
\let\phi\varphi
\let\varphi\altphi
\let\altphi\undefined

\newcommand{\abs}[1]{\left\lvert#1\right\rvert}

\DeclareMathOperator{\supp}{supp}

\newcommand\numleq[1]%
{\stackrel{\scriptscriptstyle(\mkern-1.5mu#1\mkern-1.5mu)}{=}}

\DeclareMathOperator{\de}{d}

\newfont{\tmpf}{cmsy10 scaled 2500}

\def\XXint#1#2#3{{\setbox0=\hbox{$#1{#2#3}{\int}$ }
		\vcenter{\hbox{$#2#3$ }}\kern-.6\wd0}}

\begin{document}
	
	\title{Polynomial  and horizontally polynomial functions on Lie groups}
 
	\author[G. Antonelli and E. Le Donne]{Gioacchino Antonelli and Enrico Le Donne}
\address{\textsc{Gioacchino Antonelli}: 
Scuola Normale Superiore, Piazza dei Cavalieri, 7, 56126 Pisa, Italy}
\email{gioacchino.antonelli@sns.it}
\address{\textsc{Enrico Le Donne}: 
Dipartimento di Matematica, Universit\`a di Pisa, Largo B. Pontecorvo 5, 56127 Pisa, Italy \\
University of Jyv\"askyl\"a, Department of Mathematics and Statistics, P.O. Box (MaD), FI-40014, Finland\\
\& 
Department of Mathematics, University of Fribourg, Chemin du Mus\'ee 23, 1700 Fribourg,
Switzerland
}
\email{enrico.ledonne@unifr.ch}
	
	\renewcommand{\subjclassname}{%
 \textup{2010} Mathematics Subject Classification}
 \thanks{G.A. was partially supported by the European Research Council
 	(ERC Starting Grant 713998 GeoMeG `\emph{Geometry of Metric Groups}').
 E.L.D. was partially supported by the Academy of Finland (grant
288501
`\emph{Geometry of subRiemannian groups}' and by grant
322898
`\emph{Sub-Riemannian Geometry via Metric-geometry and Lie-group Theory}')
and by the European Research Council
 (ERC Starting Grant 713998 GeoMeG `\emph{Geometry of Metric Groups}'). The authors wish to thank Mattia Calzi and Fulvio Ricci for fruitful discussions and constructive feedback around the topic of the paper.
}
	
	\maketitle
	\begin{abstract}
		We generalize both the notion of polynomial functions on Lie groups and the notion of horizontally affine maps on Carnot groups. We fix a subset $S$ of the algebra $\mathfrak g$ of left-invariant vector fields on a Lie group $\mathbb G$ and we assume that $S$ Lie generates $\mathfrak g$. We say that a function $f:\mathbb G\to \mathbb R$ (or more generally a distribution on $\mathbb G$) is 
		{\em $S$-polynomial} if for all $X\in S$ there exists $k\in \mathbb N$ such that the iterated derivative $X^k f$ is zero in the sense of distributions.
		
		First, we show that all $S$-polynomial functions (as well as distributions) are represented by analytic functions and, if the exponent $k$ in the previous definition is independent on $X\in S$, they form a finite-dimensional vector space.
		
		Second, if $\mathbb G$ is connected and nilpotent we show that  $S$-polynomial functions are polynomial functions in the sense of Leibman.
		The same result may not be true for non-nilpotent groups.
		
		Finally, we show that in connected nilpotent Lie groups, being polynomial in the sense of Leibman, being a polynomial in exponential chart, and the vanishing of mixed derivatives of some fixed degree along directions of $\mathfrak g$  are equivalent notions.
	\end{abstract}

 	\tableofcontents
 	
 	\section{Introduction}
	Following the terminology of Gromov, a {\em polarized manifold} is a (connected) manifold equipped with a choice of a subbundle of its tangent bundle, which most of the times is assumed Lie bracket generating and called {\em space of horizontal directions}. 
	In subRiemannian geometry, in analysis, but also in group theory, there are several phenomena showing that the 
	bracket generation property upgrades a ``horizontal'' property to one in every direction: this is the case, for example, of the Chow-Rashevskii Theorem, see \cite{Mon02}, the Pansu differentiability Theorem \cite{Pan89} and from a more analytic point of view of the celebrated H\"ormander Theorem \cite{Hor67}.
	In this paper we consider functions on a polarized Lie group that have the property that are polynomials along horizontal directions. 
	
	We stress that polynomial maps between groups have been studied also from an algebraic point of view, see, e.g., \cite{Pas68, Buc70, Sze85, Lei02, GT06, KP20} and references therein. Even if our interest is mainly analytic and geometric, we will also highlight the connections of our results with the algebraic point of view.

Let $\mathbb G$ be a Lie group with Lie algebra $\mathfrak g$, seen as left-invariant vector fields on $\mathbb G$.  We fix a left-Haar measure $\mu$ on $\mathbb G$.
 Let  	  $S\subseteq \mathfrak g$ be a subset that is Lie bracket generating, i.e., the only subalgebra of $\mathfrak g$ that contains $S$ is $\mathfrak g$.

We say that a distribution $f$ on $\mathbb G$ is 
		{\em $S$-polynomial} if for all $X\in S$ there exists $k\in \mathbb N$ such that the iterated derivative $X^k f$ is zero in the sense of distributions on $\mathbb G$. We say that a distribution $f$ is {\em $S$-polynomial with degree at most $k$} if for every $X\in S$ we have that $X^kf$ is zero in the sense of distributions on $\mathbb G$. For basic definitions and properties of distributions on Lie groups we refer the reader to \cref{sec:Distr}. 
The first main outcome of this paper is a regularity result for $S$-polynomial distributions on arbitrary Lie groups, see \cref{prop:FiniteDim} and \cref{proofThm1} for the proof of the following statement.
\begin{theorem}\label{thm:Intro1}
	Let $\mathbb G$ be a Lie group, let $f$ be a distribution on $\mathbb G$, and let $S\subseteq \mathfrak g$ be a subset of the Lie algebra $\mathfrak g$ that Lie generates $\mathfrak g$. If $f$ is $S$-polynomial, then it is represented by an analytic function. Moreover, the vector space of $S$-polynomial distributions with degree at most $k\in\mathbb N$ on each connected component of $\mathbb G$ is finite-dimensional. 
\end{theorem}
	It is natural to ask if an $S$-polynomial distribution on $\mathbb G$ is actually a polynomial in some sense. Various definitions of polynomial maps between groups have been proposed and studied in the literature, see \cite{Pas68} for arbitrary groups, and \cite{Buc70,Sze85} and references therein for the case of Abelian groups. A notion of {\em polynomial map} between arbitrary groups that showed to be versatile has been studied, with a special attention toward the nilpotent case, in \cite{Lei02}, see \cref{rem:LeiBravo}. In the case we deal with, i.e., the case of maps $f:\mathbb G\to\mathbb R$, Leibman's definition can be generalized for distributions. Let us define the operator $D_g$ acting on distributions $f$ on $\mathbb G$ as follows
	$$
	D_gf:=f\circ R_g-f,
	$$
	where $R_g$ stands for the right translation by $g\in\mathbb G$ and $f\circ R_g$ should be properly defined, see \eqref{eqn:fcircRg}. We say that a distribution $f$ on $\mathbb G$ is {\em polynomial à la Leibman with degree at most $d\in\mathbb N$} if
	\begin{equation}\label{eqn:IntroLeib}
	g_1,\dots,g_{d+1}\in \mathbb G \Rightarrow D_{g_1}\cdots D_{g_{d+1}}f\equiv 0, \quad \text{in the sense of distributions on $\mathbb G$}.
	\end{equation}
	We stress that our main result \cref{thm:Intro1} helps in proving that the latter notion of being polynomial, which is ``discrete'' in spirit, in our setting is equivalent to a ``differential'' one. We say that a distribution $f$ on $\mathbb G$ is {\em polynomial (in the differential sense) with degree at most $d\in\mathbb N$} if
	\begin{equation}\label{eqn:IntroDiff}
	X_1,\dots,X_{d+1}\in \mathfrak g\Rightarrow X_1\cdots X_{d+1}f\equiv 0, \quad \text{in the sense of distributions on $\mathbb G$}.
	\end{equation}
	For some benefit towards the understanding of the next result we recall that, given a Lie algebra $\mathfrak g$, the {\em nilpotent residual} $\mathfrak g_{\infty}$ of $\mathfrak g$ is the intersection $\cap_{k\in\mathbb N} \mathfrak g_k$, where $\mathfrak g:=\mathfrak g_0\supseteq \mathfrak g_1\supseteq \dots$ is the lower central series associated to $\mathfrak g$. Notice that $\mathfrak g\mathrel{/}\mathfrak g_{\infty}$ is the biggest nilpotent quotient of $\mathfrak g$. Let us denote $\mathbb G_{\infty}$ the closure of the unique connected subgroup of $\mathbb G$ with Lie algebra $\mathfrak g_{\infty}$. We call $N_{\mathbb G}:=\mathbb G\mathrel{/}\mathbb G_{\infty}$ the {\em maximal nilpotent Lie quotient}.
	The following statement is a corollary of \cref{thm:Intro1} and \cref{rem:PolyLeib}, see \cref{proof:EquIntro1}. 
	\begin{theorem}\label{prop:Equ1Intro}
		Let $\mathbb G$ be a connected Lie group. 
		\begin{enumerate}
		\item A distribution on $\mathbb G$ is polynomial à la Leibman with degree at most $d\in\mathbb N$, see \eqref{eqn:IntroLeib}, if and only if it is polynomial (in the differential sense) with degree at most $d$, see \eqref{eqn:IntroDiff}. 
		
		\item Every polynomial distribution is represented by an analytic function.
		
		\item The vector space of distributions that are polynomial with degree at most $d\in\mathbb N$ is finite-dimensional.
		
		\item For every distribution $f$ that is polynomial of degree at most $d$, and for every $X$ in the $d$-th element of the lower central series, we have $Xf\equiv 0$. Consequently, for every analytic function $f$ that is polynomial, there exists a function $\widetilde f:N_{\mathbb G}\to \mathbb R$ that is polynomial such that $\widetilde f\circ\pi=f$, where $\pi$ is the projection onto the maximal nilpotent Lie quotient $N_{\mathbb G}$.
		\end{enumerate}
		\end{theorem}
	
	Thus, with the previous result jointly with \cref{thm:Intro1}, it becomes clear that a continuous polynomial map $f:\mathbb G\to\mathbb R$ à la Leibman is analytic and in particular it is $S$-polynomial, no matter what is the choice of a Lie generating $S$. Instead, the converse may not be true in arbitrary Lie groups. The counterexample is already found in the two-dimensional group of the orientation-preserving affine functions of $\mathbb R$, called $\mathrm{Aff}^+(\mathbb R)$. With the classical choice of coordinates on $\mathrm{Aff}^+(\mathbb R)$ one can find a Lie generating $S$ such that an $S$-polynomial function that is not polynomial is $f(x,y)=(x+1)\log y$, see \eqref{ex:PositiveAffine}. 
	
	The last part of the statement of \cref{prop:Equ1Intro} tells us that polynomial maps always factor via a nilpotent group. Hence, we shall only study polynomial maps on nilpotent Lie groups. When $\mathbb G$ is a connected nilpotent Lie group $\exp:\mathfrak g\to\mathbb G$ is an analytic and surjective map, and thus one could also give another definition of ``polynomial'', namely a map $f:\mathbb G\to\mathbb R$ is {\em polynomial in exponential chart} if and only if $f\circ\exp:\mathfrak g\to\mathbb R$ is a polynomial. In case $\mathbb G$ is a connected and nilpotent Lie group, we show that the property of being $S$-polynomial propagates to the entire Lie algebra. Namely, we prove that a $S$-polynomial distribution on $\mathbb G$ is represented by a polynomial in exponential chart, and thus in particular it is $\mathfrak g$-polynomial of some degree $k\in\mathbb N$, see \cref{rem:Simpliesg}. Our second main result now can be stated as follows, see \cref{proofTHMIntro1} for the proof of the following statement.
	\begin{theorem}\label{thm:Intro2}
		Let $\mathbb G$ be a connected nilpotent Lie group, let $f$ be a distribution on $\mathbb G$, and let $S$ be a Lie generating subset of $\mathfrak g$. If $f$ is $S$-polynomial, then it is represented by a function that is polynomial in exponential chart.
	\end{theorem}
	We stress that if $\mathbb G$ is not nilpotent, being $S$-polynomial for a Lie generating $S$ may not imply being $\mathfrak g$-polynomial, or even being polynomial in exponential chart, see \eqref{ex:PositiveAffine} for such a counterexample in $\mathrm{Aff}(\mathbb R)^+$. For non-nilpotent groups we do not know either if being $\mathfrak g$-polynomial implies being polynomial, or even if being $\mathfrak g$-polynomial passes to the maximal nilpotent Lie quotient.
	
	With the previous main result we can prove that on a connected nilpotent Lie group $\mathbb G$ the different notions of being polynomial that we discussed above are equivalent and in particular they are equivalent to being $S$-polynomial for any Lie generating $S$. The following statement is a corollary of \cref{thm:Intro2}, \cref{prop:Equ1Intro}, and \cref{prop:POLYKPOLY}, see \cref{proof:Equ2Intro}.
	\begin{corollary}\label{prop:Equ2Intro}
		Let $\mathbb G$ be a connected nilpotent Lie group, and let $f$ be a distribution on $\mathbb G$. Then the following are equivalent 
		\begin{enumerate}
			\item $f$ is an $S$-polynomial distribution for some Lie generating $S\subseteq \mathfrak g$,
			\item $f$ is an $S$-polynomial distribution for all $S\subseteq \mathfrak g$,
			\item $f$ is represented by a function that is polynomial in exponential chart,
			\item $f$ is a polynomial distribution (in the differential sense), see \eqref{eqn:IntroDiff}.
			\item $f$ is a polynomial distribution à la Leibman, see \eqref{eqn:IntroLeib}.
		\end{enumerate}
	\end{corollary}
	\vspace{0.3cm} 
	
	A characterization of polynomial functions on stratified groups has been provided also in \cite{BLU07}, and we stress that our result is stronger than this characterization in \cite[Corollary 20.1.10]{BLU07}, see \cref{rem:BLUeKP}. We also notice that a characterization of harmonic polynomials in stratified groups through Almgren's frequency function has been provided in \cite[Theorem 9.1]{GR15}.
	
	Let us discuss the strategy of the proofs of \cref{thm:Intro1} and \cref{thm:Intro2}. The starting idea is to prove the results first when $f$ is smooth, and then we recover the general statements for distributions by using convolutions with smoothing kernels. The latter strategy can be performed since smooth functions that are $S$-polynomials with degree at most $k\in\mathbb N$ form a finite-dimensional vector space, see \cref{prop:FiniteDim}.
	
	The idea for proving the statements when $f$ is smooth is to somehow propagate the information about being $S$-polynomial to the directions not in $S$. Let us give here, in a particular case, a hint of how we do this. Let $f$ be smooth and polynomial with degree at most $2$ with respect to $X,Y\in\mathfrak g$, and let us denote with $e$ the identity element of $\mathbb G$. We have the following equalities for $t,s\in\mathbb R$,
	\begin{equation}\label{eqn:ReprIntro}
	\begin{split}
	f(\exp(tX)\exp(sY))&=f(\exp(tX))+s(Yf)(\exp(tX))\\ 
	&=f(e)+t(Xf)(e)+s\frac{\de}{\de\varepsilon}_{|_{\varepsilon=0}}f(\underbrace{\exp(tX)\exp(\varepsilon Y)\exp(-tX)}_{p_\varepsilon}\exp(tX)) \\
	&=f(e)+t(Xf)(e)+s\frac{\de}{\de\varepsilon}_{|_{\varepsilon=0}}\left(f(p_{\varepsilon})+t(Xf)(p_{\varepsilon})\right) \\
	&=f(e)+t(Xf)(e)+s(\mathrm{Ad}_{\exp(tX)}Y)f(e)+st(\mathrm{Ad}_{\exp(tX)}YX)f(e),
	\end{split}
	\end{equation}
	where in the first, second and third equalities we are using that $f$ is 2-polynomial with respect to $Y$ and $X$, respectively, and in the last equality we are using the definition of the adjoint, see \cref{sec:Conv}.
	In the general case, the proper generalization of the latter representation formula, see \eqref{eqn:YEAHBISAdapted}, allows to conclude the following statement: given $\ell\in\mathbb N$, the function $f$ along the concatenation of $\ell$ horizontal curves $\exp(t_1Y_1)\cdots\exp(t_\ell Y_\ell)$, with $Y_1,\dots,Y_\ell\in S$, is analytic in $(t_1,\dots,t_\ell)$ and depends only on the value of the jet of order $\delta_1$ of $f$ at the identity, where $\delta_1$ depends only on $\ell$ and the order of polynomiality of $f$ along $Y_1,\dots,Y_\ell$. The latter observation is enough to conclude both the fact that $f$ is analytic and the finite-dimensional result of \cref{thm:Intro1}. To this aim, it is important that in every polarized Lie group there exists a chart, around every point $p$, that can be written as the concatenation of a fixed number (twice the dimension of the group) of flow lines, starting from $p$, of horizontal vector fields, see \cref{lem:ConnectionBIS}.
	
	In the particular case in which $\mathbb G$ is nilpotent, the representation formula \eqref{eqn:ReprIntro} gives an additional piece of information because $\mathrm{Ad}_{\exp(tX)}$ is a ``polynomial in $t$'' sum of operators, see \eqref{eqn:AdjointExponential}. Thus, in case $\mathbb G$ is nilpotent, one concludes, from the proper generalization of \eqref{eqn:ReprIntro}, see \cref{lem:EST1BIS}, the following statement: given $\ell\in\mathbb N$, the function $f$ along the concatenation of $\ell$ horizontal curves $\exp(t_1Y_1)\cdots\exp(t_\ell Y_\ell)$, with $Y_1,\dots,Y_\ell\in S$, is a polynomial in $(t_1,\dots,t_\ell)$ with degree at most $\delta_2$, where $\delta_2$ depends only on $\ell$ and the order of polynomiality of $f$ along $Y_1,\dots,Y_\ell$, and where the coefficients of the polynomial only depend on some mixed derivatives of $f$ of bounded order at the identity.
	
	In order to prove \cref{thm:Intro2}, one first reduces to the case when $\mathbb G$ is simply connected by passing to the universal cover. Then, the first idea is to lift the problem to free-nilpotent groups, see the proof of \cref{thm:MAINGENERAL}. Free-nilpotent groups are stratified, and stratified Lie groups are also called Carnot groups. We then notice that it is sufficient to prove \cref{thm:Intro2} in the case $\mathbb G$ is a Carnot group. Now, if $\mathbb G$ is a Carnot group, we exploit that $f$ is analytic, which we have previously obtained in the setting of arbitrary Lie groups, and a blow-up argument, which we can perform since $\mathbb G$ has a homogeneous structure, to obtain that each term in the homogeneous Taylor expansion of $f$ at the identity is $S'$-polynomial of a fixed order $k\in\mathbb N$, where $S'\subseteq S$ is Lie generating, see the proof of \cref{thm:MAINCARNOT}. Now the observation above according to which every $S'$-polynomial of order $k\in\mathbb N$ is a polynomial of bounded degree along $\exp(t_1Y_1)\cdots\exp(t_\ell Y_\ell)$ allows us to conclude that every $S'$-polynomial of order $k\in\mathbb N$ has a polynomial growth order of bounded degree at infinity: this is done by using \cref{lem:Connecting} according to which we can bound below, up to a constant, the distance of a point $p=\exp(t_1Y_1)\cdots\exp(t_\ell Y_\ell)$ from the identity with $|t_1|+\dots+|t_\ell|$. Thus the homogeneous Taylor expansion of $f$ at the identity cannot have terms of arbitrarily large order, and then $f$ is a polynomial, concluding the proof.
	\vspace{0.3cm}
	
	Let us finally discuss our initial motivation for studying such a problem. If $\mathbb G$ is a Carnot group of step $s$ with stratification $\mathfrak g=V_1\oplus\dots\oplus V_s$, we take $S=V_1$ and we take a function $f$ such that for every $X\in V_1$ we have $X^2f\equiv 0$, we recover the notion of {\em horizontally affine maps} on a Carnot group, which has recently been introduced and studied in \cite{LDMR20}. The main result of \cite{LDMR20}, i.e., \cite[Theorem 1.1]{LDMR20}, is a complete characterization of horizontally affine maps on step-2 Carnot groups. Our main result \cref{thm:Intro2} can be seen as a broad generalization of the part of the statement in \cite[Theorem 1.1]{LDMR20} according to which every horizontally affine map on a step-2 Carnot group is ultimately polynomial in exponential chart. Indeed, our \cref{thm:Intro2} holds for arbitrary connected nilpotent Lie groups, an arbitrary degree of polynomiality, and even if we ask the polynomial property with respect to a finite set $S$ that Lie generates, which is not necessarily a vector subspace of $\mathfrak g$.
	
	The interest toward horizontally affine maps on Carnot groups is motivated by the fact that they are linked to precisely monotone sets, which have been first introduced and studied in \cite{CK10}, and \cite{CKN11}, in order to show that the first Heisenberg group $\mathbb H^1$ with the subRiemannian distance does not biLipschitz embed into $L^1(\mathbb R,\mathscr{L}^1)$. A {\em precisely monotone} set $E$ in a Carnot group $\mathbb G$ is a subset of $\mathbb G$ such that $E$ and the complement $E^c$ are {\em $h$-convex}, see \cite[Definition 3.1]{Ric06}. A set $E$ is {\em $h$-convex} in $\mathbb G$ if whenever $a,b\in E$ and there exists $\gamma$ an integral curve of a horizontal left-invariant vector field with extrema $a,b$, then $\gamma$ is contained in $E$. It is simple to observe, from the very definition, that the sublevel sets of a horizontally affine map on a Carnot group are precisely monotone sets. 
	
	The problem of classifying precisely monotone sets is rather difficult already in the easiest Carnot groups. The precisely monotone sets have been completely classified in the Heisenberg groups $\mathbb H^n$, see \cite[Theorem 4.3]{CK10}, and \cite[Proposition 65]{NY18}, and in $\mathbb H^1\times\mathbb R$, see \cite[Theorem 1.2]{Mor18}. In all the latter three cases one can prove that if $E\notin\{\emptyset,\mathbb G\}$ is precisely monotone,
	then $\partial E$ is a hyperplane in exponential chart and $\mathrm{Int}(E)$ and $\overline E$ are one of the open, respectively closed, half-spaces bounded by $\partial E$. On the contrary, from the results in \cite{LDMR20}, it follows that, already in the step-2 case, there are sublevel sets of horizontally affine maps, and thus precisely monotone sets, that are not half-spaces in exponential chart. The results in \cite[Theorem 1.1]{LDMR20} show  that in arbitrary Carnot groups there are plenty of horizontally affine maps that are not affine, and we could construct plenty of precisely monotone sets that are not half-spaces in exponential chart. On the other hand, since we proved in particular that in arbitrary Carnot groups horizontally affine maps are polynomial in exponential chart, the precisely monotone sets constructed as sublevel sets of horizontally affine maps result in being semialgebraic sets in exponential chart. Also we stress that in some Carnot groups, e.g., the free Carnot group of step 3 and rank 2, there are precisely monotone sets whose boundary is not an algebraic variety in exponential chart, see \cite[Theorem 6.2]{BLD19}. As a consequence, in general it is even not true that the boundary of a precisely monotone set in a Carnot group is a zero-level set of a horizontally affine function.
	\vspace{0.3cm}
	
	The structure of the paper is as follows.
	
	In \cref{sec:Prel} we recall some basic facts about Lie groups. In particular, in \cref{sec:Conv} we discuss the convolution in arbitrary Lie groups. In \cref{sec:Distr} we recall basic facts about distributions on arbitrary Lie groups. In \cref{sec:NilpAndCarnot} we fix the notation on nilpotent and stratified, also called Carnot, groups. In \cref{sec:3} we prove \cref{thm:Intro1}. In particular, in \cref{sec:3.1} we introduce the definition of $S$-polynomial distributions and polynomial distributions in arbitrary Lie groups and we study how $S$-polynomiality behaves under pointwise convergence. In \cref{sec:3.2} we prove \cref{lem:Representation} that will allow to prove the representation formula in  \cref{lem:EST1BISAdapted}. In \cref{sec:3.3} we conclude the proof of \cref{thm:Intro1}. In \cref{sec:4} we prove \cref{thm:Intro2}. In particular, in \cref{sec:4.1} we reduce to the case of Carnot groups and then we conclude by exploiting \cref{lem:EST1BIS}, a consequence of \cref{lem:Representation}. In \cref{sec:App} we give the proof of \cref{prop:Equ1Intro} and of \cref{prop:Equ2Intro}. Finally, in \cref{sec:examples}, we discuss some examples.
 	\section{Preliminaries}\label{sec:Prel}
 	
 	\subsection{Adjoint and convolutions on Lie groups}\label{sec:Conv}
 	We recall here the definition of the adjoint map on an arbitrary Lie group, mostly to fix notation and have formulas ready for later.
 	\begin{definition}[Conjugate and adjoint]
 		Let $\mathbb G$ be a Lie group with identity $e$. Let us fix $g\in\mathbb G$ and define the {\em conjugate function} $C_g:\mathbb G\to \mathbb G$ as $C_g(h):=ghg^{-1}$,
 		and the {\em adjoint operator} $\mathrm{Ad}_g:\mathfrak g\to \mathfrak g$ as
 		$
 		\mathrm{Ad}_g(X):=(\de (C_g)_e)(X).
 		$
 	\end{definition}
 	\begin{remark}[Formula for the adjoint]\label{rem:Adjoint}
 		In an arbitrary Lie group $\mathbb G$ the following formula holds
 		\begin{equation}\label{eqn:AdjointExponential}
 		\mathrm{Ad}_{\exp(X)}(Y)=\sum_{j=0}^{+\infty}\frac{1}{j!}(\mathrm{ad}_X)^jY,
 		\end{equation}
 		where the operator $\mathrm{ad}_X:\mathfrak g\to \mathfrak g$ is defined as $\mathrm{ad}_X(Y):=[X,Y]$, see \cite[Equation (3) page 12]{CG90}. Notice that if $\mathbb G$ is nilpotent then \eqref{eqn:AdjointExponential} is a finite sum. 
 	\end{remark}
 	
 	We recall the notion of {\em modular function} on a Lie group $\mathbb G$. We follow closely the presentation in \cite[Chapter 2]{Fol95}. We fix a {\em left-Haar measure} $\mu$ on $\mathbb G$.
 	Recall that every two left-invariant Haar measures are one a constant multiple of the other, see \cite[Theorem 2.20]{Fol95}. For every $x\in\mathbb G$ the measure defined by the equality $\mu_x:=(R_{x^{-1}})_*\mu$, where $R_g$ stands for the right translation by $g$, is a left-invariant Haar measure. Thus there exists $\Delta(x)>0$ such that $\mu_x=\Delta(x)\mu$. The function $x\mapsto \Delta(x)$ is called {\em modular function} and it enjoys the following properties
 	\begin{itemize}
 		\item $\Delta:\mathbb G\to\mathbb R$ is an analytic function and
 		\begin{equation}\label{ModularSmooth}
 		\Delta(x)=\det\mathrm{Ad}_x,
 		\end{equation}
 		for every $x\in\mathbb G$, see \cite[Proposition 2.24 and Proposition 2.30]{Fol95};
 		\item we have
 		\begin{equation}\label{ChangeVariables}
 		(\mathrm{inv})_*\mu=(\Delta\circ\mathrm{inv})\mu,
 		\end{equation}
 		where $\mathrm{inv}$ stands for the inverse function on $\mathbb G$, see \cite[Equation (2.32)]{Fol95}.
 	\end{itemize}
 	A Lie group is {\em unimodular} if $\Delta\equiv 1$, i.e., if $\mu$ is also a right-invariant Haar measure. Notice that every connected nilpotent Lie group is unimodular, due to \eqref{ModularSmooth}. We introduce now the convolution between two functions.
 	\begin{definition}[Convolution]\label{def:Convolution}
 		Let $\mu$ be a left-invariant Haar measure on a Lie group $\mathbb G$. Let $f,g\in L^1(\mathbb G,\mu)$. 
 		The {\em convolution} of $f,g:\mathbb G\to\mathbb R$ is defined as
 		\begin{equation}\label{eqn:ConvFuncFunc}
 		f\ast g(x):=\int_{\mathbb G} f(y)g(y^{-1}x)\de\mu(y), \qquad \forall x\in\mathbb G.
 		\end{equation}
 		The definition is well-posed since an application of Fubini theorem implies that the integral is absolutely convergent for every $x\in\mathbb G$, see \cite[page 50]{Fol95}.
 	\end{definition}
 	It is readily seen by the very \cref{def:Convolution} that when $g$ is continuous with compact support, the convolution $f\ast g$ is continuous. Moreover, if $g\in C^{\infty}_c(\mathbb G)$ we get that $X(f\ast g)=f\ast Xg$ for every left-invariant vector field $X$. Thus if $g\in C^{\infty}_c(\mathbb G)$ we conclude that $f\ast g\in C^{\infty}(\mathbb G)$. In addition one has that $\mathrm{supp}(f\ast g)\subseteq \overline{\mathrm{supp}(f)\cdot \mathrm{supp}(g)}$, and then if $f$ and $g$ have compact support, then also $f\ast g$ has compact support. 
 	
 	The following change of variable formula holds, taking into account \cite[Equation (2.36)]{Fol95},
 	\begin{equation}\label{eqn:TrueConv}
 	f\ast g(x):=\int_{\mathbb G} f(y)g(y^{-1}x)\de\mu(y)=\int_{\mathbb G} f(xy^{-1})g(y)\Delta(y^{-1})\de\mu(y),
 	\end{equation}
 	whenever $f,g\in L^1(\mathbb G,\mu)$ and $x\in\mathbb G$, since \eqref{ChangeVariables} holds. Let us introduce a variant of the convolution $\ast$ introduced before. Given $f,g\in L^1(\mathbb G,\mu)$, we define 
 	\begin{equation}\label{eqn:OtherConvolution}
 	f\,\widetilde{\ast}\, g(x):=\int_{\mathbb G} f(y)g(yx)\de\mu(y)=\int_{\mathbb G}f(yx^{-1})g(y)\Delta(x^{-1})\de\mu(y),
 	\end{equation}
 	where the first integral is absolutely convergent since $f,g\in L^1(\mathbb G,\mu)$ and the second formula holds taking $(R_{x})_*\mu=\Delta(x^{-1})\mu$ into account. It is readily seen that, when $f$ and $g$ have compact support, then $f\,\widetilde{\ast}\, g$ has compact support and if $g\in C^{\infty}_c(\mathbb G)$ we can write $X(f\,\widetilde{\ast}\, g)=f\,\widetilde{\ast}\, Xg$ whenever $X$ is a left-invariant vector field. As a consequence if $g\in C^\infty_c(\mathbb G)$ then $f\,\widetilde{\ast}\, g\in C^\infty(\mathbb G)$. Let us notice the following fact, which comes from \eqref{ChangeVariables}, $(R_{x^{-1}})_*\mu=\Delta(x)\mu$, \eqref{eqn:TrueConv} and \eqref{eqn:OtherConvolution}
 	\begin{equation}\label{eqn:astcheck}
 	\text{$f\,\widetilde{\ast}\, g =\check{(f\Delta)}\ast g$, where, for a function $h:\mathbb G\to\mathbb R$, we define $\check{h}:=h\circ\mathrm{inv}$.}
 	\end{equation}
 	We recall that, when the following integral makes sense for functions $f,g:\mathbb G\to\mathbb R$, we denote
 	\begin{equation}\label{eqn:ScalarProduct}
 	\langle f,g\rangle :=\int_{\mathbb G}f(x)g(x)\de\mu(x).
 	\end{equation}
 	We claim that the following formula holds
 	\begin{equation}\label{eqn:AdjointConv}
 	\langle f\ast h,g\rangle =\langle h,f\,\widetilde{\ast}\, g\rangle=\langle h,\check{(f\Delta)}\ast g\rangle, \quad \text{for every $f,g,h$ continuous with compact support.}
 	\end{equation}
 	Indeed, taking \eqref{eqn:astcheck} into account, by Fubini theorem and the fact that $\mu$ is left-invariant we have
 	\begin{equation}
 	\begin{split}
 	\langle f\ast h,g\rangle &= \int_{\mathbb G}\left(\int_{\mathbb G}f(y)h(y^{-1}x)\de\mu(y)\right)g(x)\de\mu(x) \\
 	&= \int_{\mathbb G}\left(\int_{\mathbb G}g(x)h(y^{-1}x)\de\mu(x)\right)f(y)\de\mu(y)  \\
 	&= \int_{\mathbb G}\left(\int_{\mathbb G}g(yx)h(x)\de\mu(x)\right)f(y)\de\mu(y)  \\
 	&= \int_{\mathbb G}\left(\int_{\mathbb G}f(y)g(yx)\de\mu(y)\right)h(x)\de\mu(x) =\langle h,f\,\widetilde{\ast}\, g\rangle=\langle h,\check{(f\Delta)}\ast g\rangle.
 	\end{split}
 	\end{equation}
 	Notice that if we are on a unimodular Lie group, \eqref{eqn:AdjointConv} implies the well-known formula $\langle f\ast h,g\rangle =\langle h,\check{f}\ast g\rangle$ when the integral makes sense.
 	
 	Let us end this section by computing, for every left-invariant vector field $X$ on $\mathbb G$, the action of the adjoint $X^\top$ of $X$ with respect to $\langle \cdot,\cdot\rangle$, on the smooth functions with compact support. We recall that if $\phi\in C^{\infty}_c(\mathbb G)$, the operator $X^\top$ acting on $\phi$ is the unique smooth function $X^\top\phi$ such that 
 	$$
 	\langle X^\top\phi,\psi\rangle=\langle \phi,X\psi\rangle, \qquad \text{for all $\psi\in C^{\infty}_c(\mathbb G)$}.
 	$$
 	For every $f,g\in C^{\infty}_c(\mathbb G)$, every $X\in\mathfrak g$ and $t>0$ the following equality holds 
 	$$
 	\int _{\mathbb G} f(y\exp(tX))g(y)\de\mu(y)=\int_{\mathbb G} f(y)g(y\exp(-tX))\Delta(\exp(-tX))\de\mu(y),
 	$$
 	due to the fact that $(R_{x^{-1}})_*\mu=\Delta(x)\mu$. We deduce that $X^\top = -\Delta(e)X-(X\Delta)(e)\mathrm{id}=-X-(X\Delta)(e)\mathrm{id}$, since $\Delta(e)=1$. Namely
 	\begin{equation}\label{eqn:Adjoint}
 	\langle Xf,g\rangle=:\langle f,X^\top g\rangle =\langle f, -Xg-(X\Delta)(e)g\rangle, \quad \text{for every $f,g\in C^\infty_c(\mathbb G)$}.
 	\end{equation}

 	\subsection{Distributions on Lie groups}\label{sec:Distr}
 	In this section we recall some basic facts about the theory of distributions on a Lie group $\mathbb G$, see also \cite{Ehr56}. We stress that we will use a variant of the definition of convolution between functions and distributions with respect to \cite{Ehr56}. We fix a left-invariant Haar measure $\mu$ on $\mathbb G$. We remark that every vector space considered will be a vector space over $\mathbb R$.
 	
 	We denote with $\mathcal{D}(\mathbb G)\equiv C^{\infty}_c(\mathbb G)$ the topological vector space of real-valued $C^{\infty}$ functions with compact support on $\mathbb G$ equipped with the final locally convex topology with respect to the immersions $C^{\infty}_K(\mathbb G)\hookrightarrow \mathcal{D}(\mathbb G)$, where $K$ is a compact subset of $\mathbb G$ and $C^{\infty}_K(\mathbb G)$ is the space of real-valued $C^{\infty}$ functions with compact support contained in $K$. Let us recall that, if we fix $\{X_1,\dots,X_n\}$ a basis of the Lie algebra $\mathfrak g$, the (countably many) seminorms that induce the locally convex topology on $C^{\infty}_K(\mathbb G)$ are
 	$$
 	\|f\|_{j_1,\dots,j_r,K}:=\sup_{x\in K}|X_{j_1}\dots X_{j_r} f(x)|,
 	$$ 
 	where $r\in\mathbb N$ and $j_1,\dots,j_r\in\{1,\dots,n\}$. Let us denote with $\mathcal{D}'(\mathbb G)$ the dual of $\mathcal{D}(\mathbb G)$, i.e., the set of continuous linear functionals from $\mathcal{D}(\mathbb G)$ to $\mathbb R$, equipped with the locally convex weak* topology. If $\Phi\in\mathcal{D}'(\mathbb G)$, we recall that by $\langle \Phi,\phi\rangle$ we mean the evaluation of $\Phi$ at $\phi\in \mathcal{D}(\mathbb G)$. Let us remark that if $f\in L^1_{\mathrm{loc}}(\mathbb G,\mu)$, there is a canonical way of seeing $f$ as a distribution, by means of 
 	$$
 	\langle f,\phi\rangle :=\int_{\mathbb G}f(x)\phi(x)\de\mu(x), \qquad \text{for all $\phi\in \mathcal{D}(\mathbb G)$},
 	$$
 	so that the notation of the evaluation $\langle\cdot,\cdot\rangle$ for a distribution is consistent with the one introduced in \eqref{eqn:ScalarProduct}.
 	
 	We recall that if $f_j\in\mathcal{D}(\mathbb G)$, and $f\in\mathcal{D}(\mathbb G)$, then $f_j\to f$ in $\mathcal{D}(\mathbb G)$ if and only of there exists a compact subset $K\subseteq \mathbb G$ such that all the supports of $f_j$'s are contained in $K$ and for every $r\in\mathbb N$  and every $j_1,\dots,j_r\in\{1,\dots,n\}$ we have $X_{j_1}\dots X_{j_r}f_j\to X_{j_1}\dots X_{j_r}f$ uniformly on $K$. On the other hand if $\Phi_j,\Phi\in\mathcal{D}'(\mathbb G)$, we have that $\Phi_j\to\Phi$ in the weak* topology of $\mathcal{D}'(\mathbb G)$ if and only if $\langle\Phi_j,f\rangle\to\langle \Phi,f\rangle$ for every $f\in\mathcal{D}(\mathbb G)$.   

 	\begin{definition}[Derivative of a distribution and convolution with a function]\label{def:DerivativeAndConv}
 		If $X$ is a left-invariant vector field on $\mathbb G$ and $\Phi\in\mathcal{D}'(\mathbb G)$ we define 
 		$$
 		\langle X\Phi,g\rangle := \langle \Phi, X^\top g\rangle, \qquad \forall g\in \mathcal{D}(\mathbb G),
 		$$
 		where the action of the adjoint $X^\top$ of $X$ on $\mathcal{D}(\mathbb G)$ has been explicitly computed in \eqref{eqn:Adjoint}.
 		
 		Moreover, if $f\in\mathcal{D}(\mathbb G)$ and $\Phi\in\mathcal{D}'(\mathbb G)$ we define 
 		\begin{equation}\label{eqn:Convolution}
 		\langle f\ast \Phi, g\rangle := \langle \Phi, f\,\widetilde{\ast}\, g\rangle=\langle \Phi,\check{(f\Delta)}\ast g\rangle, \quad \forall g\in\mathcal{D}(\mathbb G),
 		\end{equation}
 		where $f\,\widetilde{\ast}\, g$ is defined in \eqref{eqn:OtherConvolution}, see also \eqref{eqn:astcheck}.
 	\end{definition}
 	\begin{remark}[Derivative and convolution of a distribution]\label{rem:DerivativeAndConvolution}
 		We notice that the definitions given in \cref{def:DerivativeAndConv} are consistent with the case in which $\Phi$ is a function, see \eqref{eqn:Adjoint} and \eqref{eqn:AdjointConv}. Moreover, if $\{g_n\}_{n\in\mathbb N}\subseteq \mathcal{D}(\mathbb G)$ and $g_n\to g\in \mathcal{D}(\mathbb G)$ in the topology of $\mathcal{D}(\mathbb G)$, then $X^\top g_n\to X^\top g$ and $f\,\widetilde{\ast}\, g_n\to f\,\widetilde{\ast}\, g$ in the topology of $\mathcal{D}(\mathbb G)$, due to the explicit expressions in \eqref{eqn:Adjoint}, and \eqref{eqn:OtherConvolution}, respectively. Thus $X\Phi$ and $f\ast \Phi$ are well defined distributions whenever $\Phi\in\mathcal{D}'(\mathbb G)$, $f\in\mathcal{D}(\mathbb G)$, and $X$ is a left-invariant vector field. 
 		
 		Actually, we claim that for every $f\in\mathcal{D}(\mathbb G)$ and every $\Phi\in\mathcal{D}'(\mathbb G)$, the distribution $f\ast\Phi$ coincides with the $C^{\infty}(\mathbb G)$ function defined as follows
 		$$
 		f\ast\Phi(x):=\langle \Phi, y\mapsto f(xy^{-1})\Delta(y^{-1})\rangle.
 		$$
 		Indeed, it can be verified\footnote{Compare with \cite[Proposition 6]{Ehr56}. Let us call $\tau^f_x(y):=f(xy^{-1})\Delta(y^{-1})$. It is sufficient to prove that, if $f\in\mathcal{D}(\mathbb G)$ is fixed, the map $x\mapsto \tau_x^f$ is continuous from $\mathbb G$ to $\mathcal{D}(\mathbb G)$ and that $t^{-1}(\tau_{\exp(tX)x}^f-\tau_x^f)\to \tau^{X^Rf}_x$ in $\mathcal{D}(\mathbb G)$ as $t$ goes to zero.}  that $f\ast\Phi(x)$ is $C^{\infty}(\mathbb G)$ and for every $h\in \mathcal{D}(\mathbb G)$ we have
 		$$
 		\int_{\mathbb G} f\ast\Phi(x)h(x)\de\mu(x) = \langle \Phi, y\mapsto \int_{\mathbb G}f(xy^{-1})h(x)\Delta(y^{-1})\de\mu(x)\rangle =\langle \Phi, f\,\widetilde{\ast}\, h\rangle=\langle \Phi,\check{(f\Delta)}\ast g\rangle,
 		$$
 		where the last two equalities follow from \eqref{eqn:OtherConvolution} and \eqref{eqn:astcheck}, respectively.

 	\end{remark}
 	\begin{remark}[Derivative of a convolution]
 		Let us prove that if $f\in\mathcal{D}(\mathbb G)$, $\Phi\in\mathcal{D}'(\mathbb G)$ and $X\in\mathfrak g$ we have 
 		\begin{equation}\label{rem:DerivativeOfAConvolution}
 		X(f\ast \Phi)=f\ast (X\Phi).
 		\end{equation}
 		Indeed, for every $g\in\mathcal{D}(\mathbb G)$,
 		\begin{equation}
 		\begin{split}
 		\langle X(f\ast\Phi),g\rangle &=\langle f\ast\Phi,X^\top g\rangle =\langle \Phi, f\,\widetilde{\ast}\, (X^\top g)\rangle  \\
 		&= \langle \Phi, X^\top(f\,\widetilde{\ast}\, g)\rangle = \langle X\Phi,f\,\widetilde{\ast}\, g\rangle = \langle f\ast X\Phi,g\rangle,
 		\end{split}
 		\end{equation}
 		where we only exploited \cref{def:DerivativeAndConv} and in the third equality we used the fact that $X^\top = -X-(X\Delta)(e)\mathrm{id}$, see \eqref{eqn:Adjoint}, and $X(f\,\widetilde{\ast}\, g)=f\,\widetilde{\ast}\,(Xg)$. 
 	\end{remark}
 	\begin{definition}[Approximate identity]\label{def:approximateidentity} 
 		We say that a sequence $\{\phi_n\}_{n\in \mathbb N}\subseteq \mathcal{D}(\mathbb G)$ is an {\em approximate identity} if the following holds
 		\begin{equation}\label{eqn:Convo1}
 		\phi_n\ast \Phi\to \Phi, \quad \mbox{in the topology of $\mathcal{D}'(\mathbb G)$ for all $\Phi\in\mathcal{D}'(\mathbb G)$.}
 		\end{equation}
 	\end{definition}
 		We now construct an approximate identity on a Lie group $\mathbb G$. On $\mathbb G$ we fix an arbitrary right-invariant distance $d$ that induces the manifold topology, which for example can be taken to be Riemannian. We fix $\mu$ a left-invariant Haar measure on $\mathbb G$ as well. We know that $\exp:\mathfrak g\to \mathbb G$ is a local analytic diffeomorphism around $0\in\mathfrak g$, see \cite[page 11]{CG90}. By a classical choice of smooth functions that are compactly supported in a sufficiently small neighbourhood of $0\in \mathfrak g$, and by reading these functions on $\mathbb G$ through the exponential map, we can readily build a family of positive functions $\{\phi_{n}\}_{n\in\mathbb N}\subseteq \mathcal{D}(\mathbb G)$ such that the following conditions hold
 		\begin{enumerate}
 			\item $\int_{\mathbb G} \phi_n(y)\de\mu(y)=1$, for all $n\in\mathbb N$;
 			\item The family $\{\mathrm{supp}(\phi_n)\}_{n\in\mathbb N}$ is a fundamental system of compact neighbourhoods of the identity element $e\in\mathbb G$ contained in a common compact neighbourhood of $e$. 
 		\end{enumerate}
 		\begin{proposition}
 			With the notation and the setting above, whenever $f\in\mathcal{D}(\mathbb G)$, it holds $\phi_n\,\widetilde{\ast}\, f\to f$ in the topology of $\mathcal{D}(\mathbb G)$, where $\,\widetilde{\ast}\,$ is defined in \eqref{eqn:OtherConvolution}. Thus, whenever $\Phi\in\mathcal{D}'(\mathbb G)$, we conclude $\phi_n\ast \Phi\to \Phi$ in the topology of $\mathcal{D}'(\mathbb G)$ and then $\{\phi_n\}_{n\in\mathbb N}$ is an approximate identity.
 		\end{proposition}
 		\begin{proof}
 		By the definitions of convolution with a distribution, see \eqref{eqn:Convolution}, and of convergence in $\mathcal{D}'(\mathbb G)$, the second part of the statement readily follows from the first part.
 		
 		Hence, let us fix $f\in\mathcal{D}(\mathbb G)$ and let $K$ be the support of $f$. Since $\{\mathrm{supp}(\phi_n)\}_{n\in\mathbb N}$ is a fundamental system of compact neighbourhoods of the identity element $e\in\mathbb G$ there exists an infinitesimal decreasing sequence $\varepsilon_n\to 0$ such that 
 		\begin{equation}\label{eqn:E1}
 		\text{if $x\in\mathbb G$ is such that $d(x,e)>\varepsilon_n$ then $\phi_n(x)=0$}. 
 		\end{equation}
 		Moreover, since $\supp(\phi_n\,\widetilde{\ast}\, f)\subseteq \overline{\supp(\phi_n)^{-1}\cdot\supp f}$, there exists a compact set $\widetilde K\supseteq K$ such that
 		\begin{equation}\label{eqn:support}
 		\mathrm{supp}(f),\mathrm{supp}(\phi_n\,\widetilde{\ast}\, f)\subseteq \widetilde K, \qquad \forall n\in\mathbb N.
 		\end{equation}
 		Moreover there exists a compact set $K'\supseteq\widetilde K$ such that 
 		\begin{equation}\label{eqn:E2}
 		\text{$\supp(\phi_n)\cdot \widetilde K\subseteq K'$}, \qquad \forall n\in\mathbb N.
 		\end{equation}
 		Let us fix $x\in \widetilde K$. Then from the fact that for all $n\in\mathbb N$ it holds $\int_{\mathbb G}\phi_n(y)\de\mu (y)=1$, and by the very definition of the convolution $\,\widetilde{\ast}\,$ we get, for all $n\in\mathbb N$, that the following inequality holds
 		\begin{equation}\label{eqn:STIMA}
 		\begin{aligned}
 		|\phi_n\,\widetilde{\ast}\, f(x)- f(x)| &= \left|\int_{\mathbb G}\phi_n(y)(f(yx)-f(x))\de\mu(y)\right| \\
 		&\leq \int_{\{y\in \mathbb G:y\in\supp(\phi_n)\}}\phi_n(y)|f(yx)-f(x)|\de\mu(y).
 		\end{aligned}
 		\end{equation}
 		Now, from \eqref{eqn:E2}, we get that whenever $x\in\widetilde K$ and $y\in\supp(\phi_n)$ we have $\{x,yx\}\subseteq K'$. Since $f$ is uniformly continuous on the compact set $K'$, taking into account also \eqref{eqn:E1}, we get that for every $\varepsilon>0$ there exists $n_0=n_0(\varepsilon)$ such that for every $n\geq n_0$ the following holds
 		\begin{equation}\label{eqn:COnt}
 		\text{if $x\in\widetilde K$ and $y\in \supp(\phi_n)$, then $d(yx,x)=d(y,e)\leq \varepsilon_n$ and thus $|f(yx)-f(x)|\leq\varepsilon$,}
 		\end{equation}
 		where we used that $d$ is right-invariant.
 		Thus if we fix $\varepsilon >0$ we can use \eqref{eqn:COnt} and continue \eqref{eqn:STIMA} in order to obtain that for all $x\in \widetilde K$ and for every $n\geq n_0(\varepsilon)$ the following inequality holds
 		$$
 		|\phi_n\,\widetilde{\ast}\, f(x)-f(x)| \leq \int_{\{y\in \mathbb G:y\in\supp(\phi_n)\}}\phi_n(y)\varepsilon\de\mu(y) \leq \int_{\mathbb G}\phi_n(y)\varepsilon\de\mu(y)=\varepsilon.
 		$$
 		The previous equation implies that 
 		$$
 		\lim_{n\to+\infty}\sup_{x\in \widetilde K}|\phi_n\,\widetilde{\ast}\, f(x)-f(x)| = 0.
 		$$
 		Let us fix $\{X_1,\dots,X_n\}$ a basis of $\mathfrak g$. By exploiting the fact that $X(f\,\widetilde{\ast}\, g)=f\,\widetilde{\ast}\, (Xg)$ for every left-invariant vector field $X$ and every $f,g\in\mathcal{D}(\mathbb G)$, we get arguing exactly as before, that for every multi-index $\alpha:=(j_1,\dots,j_r)$ with $r\in\mathbb N$ and  $j_1,\dots,j_r\in\{1,\dots,n\}$, 
 		\begin{equation}\label{eqn:limit}
 		\lim_{n\to+\infty}\sup_{x\in \widetilde K}|X^\alpha(\phi_n\,\widetilde{\ast}\, f)(x)-X^\alpha f(x)| = \lim_{n\to+\infty}\sup_{x\in \widetilde K}|(\phi_n\,\widetilde{\ast}\, X^\alpha f)(x)-X^\alpha f(x)|=0,
 		\end{equation}
 		where $X^\alpha := X_{j_1}\dots X_{j_r}$. Hence \eqref{eqn:support} and \eqref{eqn:limit} precisely means that $\phi_n\,\widetilde{\ast}\, f\to f$ in the topology of $\mathcal{D}(\mathbb G)$ that was what we wanted to prove. 
 		\end{proof}
 		We define now the precomposition of a distribution with a right translation.
 		\begin{definition}[Right translation of a distribution]\label{def:DgOnDistributions}
 			Let $f$ be a distribution on a Lie group $\mathbb G$ and let $R_g$ denote the right translation by $g\in\mathbb G$. Then we define 
 			\begin{equation}\label{eqn:fcircRg}
 			\langle f\circ R_g,\phi\rangle := \langle f, \Delta(g^{-1})\phi\circ R_{g^{-1}}\rangle, \qquad \text{for all $\phi\in \mathcal{D}(\mathbb G)$}.
 			\end{equation}
 		\end{definition}
 	By using the fact that $(R_{g})_*\mu=\Delta(g^{-1})\mu$ for all $g\in\mathbb G$, we see that the above definition is consistent with \eqref{eqn:ScalarProduct} when $f$ is a smooth function.
 	Moreover, let us define, for every $g\in\mathbb G$, the operator $D_g$ acting on distributions $f$ on $\mathbb G$ as follows
 	\begin{equation}\label{eqn:ActionDgDistr}
 	D_gf:=f\circ R_g-f.
 	\end{equation}
 	\begin{remark}[Commutation of convolution and right translation]\label{rem:ConvolutionDg}
 		We claim that if $\phi\in\mathcal{D}(\mathbb G)$, $f\in\mathcal{D}'(\mathbb G)$ and $g\in\mathbb G$ then the following equality holds
 		\begin{equation}\label{eqn:FirstConv}
 		\langle (\phi\ast f)\circ R_g,\psi\rangle =\langle \phi\ast (f\circ R_g),\psi\rangle, \qquad \text{for all $\psi\in \mathcal{D}(\mathbb G)$}.
 		\end{equation}
 		Indeed, we have
 		\begin{equation*}
 		\begin{split}
 		\langle (\phi\ast f)\circ R_g,\psi\rangle &= \langle \phi\ast f,\Delta(g^{-1})\psi\circ R_{g^{-1}}\rangle = \langle f, \Delta(g^{-1})\phi\,\widetilde{\ast}\,(\psi\circ R_{g^{-1}})\rangle \\
 		&=\langle f,\Delta(g^{-1})(\phi\,\widetilde{\ast}\,\psi)\circ R_{g^{-1}}\rangle =\langle f\circ R_g,\phi\,\widetilde{\ast}\,\psi\rangle \\
 		&=\langle \phi\ast(f\circ R_g),\psi\rangle,
 		\end{split}
 		\end{equation*}
 		where in the first and the fourth equalities we used the definition in \eqref{eqn:fcircRg}, in the second and the fifth equalities we used the definition in \eqref{eqn:Convolution}, and in the third equality we used $\phi\,\widetilde{\ast}\,(\psi\circ R_{g^{-1}})=(\phi\,\widetilde{\ast}\,\psi)\circ R_{g^{-1}} $, which we now prove. Indeed, for every $x\in\mathbb G$,
 		$$
 		\phi\,\widetilde{\ast}\,(\psi\circ R_{g^{-1}})(x)=\int \phi(y)\psi(yxg^{-1})\de\mu(y)=(\phi\,\widetilde{\ast}\,\psi)\circ R_{g^{-1}}(x),
 		$$
 		where we used the definition in \eqref{eqn:OtherConvolution}. As a consequence of \eqref{eqn:FirstConv} we obtain the following equality, for every $\phi\in\mathcal{D}(\mathbb G)$, $f\in\mathcal{D}'(\mathbb G)$ and $g\in\mathbb G$ 
 		\begin{equation}\label{eqn:DgConvolution}
 		D_g(\phi\ast f) = \phi\ast D_g f.
 		\end{equation}
 	\end{remark}
 			
 	\subsection{Nilpotent Lie groups and stratified groups}\label{sec:NilpAndCarnot}
 	We now focus the attention on special classes of Lie groups. First we discuss nilpotent groups, and then we discuss stratified groups, also called Carnot groups. For the notions in this section we refer the reader to classical books and recent surveys or lecture notes, e.g., \cite{FS82, Varadarajan, CG90, BLU07, LD17}. We stress that every vector space (or Lie algebra) considered will be a vector space (or Lie algebra) over $\mathbb R$. 
 	We recall that a connected Lie group is nilpotent if and only if its algebra is nilpotent. It is well known that the exponential map $\exp\colon\mathfrak g\to \mathbb G$ is a global analytic diffeomorphism whenever $\mathbb G$ is a simply connected nilpotent Lie group, while in general, if $\mathbb G$ is connected and nilpotent but not necessarily simply connected, it is an analytic and surjective map, see \cite[Theorem 3.6.1]{Varadarajan}.

 	We say that a Lie algebra $\mathfrak g$ is {\em stratifiable} if there exist a {\em stratification} of it, namely there exist subspaces $V_1,\dots, V_s$ of the Lie algebra $\mathfrak g$ such that
 	\[
 	\mathfrak g = V_1\oplus \dots\oplus V_s,\qquad [V_1,V_j]=V_{j+1} \quad\forall j=1,\dots,s-1,\qquad [V_1,V_s]=\{0\}.
 	\]
 	When one of such stratifications is fixed we say that $\mathfrak g$ is a {\em stratified} Lie algebra. Recall that two stratifications of a stratifiable Lie algebra differ by an automorphism, see \cite[Proposition 2.17]{LD17}. Notice that every stratified Lie algebra is nilpotent. A {\em stratified group} (also called a {\em Carnot group}) $\mathbb G$ is a simply connected Lie group whose Lie algebra $\mathfrak g$ is stratifiable and one such stratification is fixed. If the stratifications of $\mathbb G$ are made with vector subspaces $V_1,\dots,V_s$, we call $s$ the {\em step} of $\mathbb G$, while $m\coloneqq\dim (V_1)$ is called \emph{rank} of $\mathbb G$. For every $i=1,\dots,s$, we call $V_i$ the {\em $i$-th layer} of the stratification.
 	
 	Every Carnot group has a one-parameter family of {\em dilations} that we denote by $\{\delta_\lambda: \lambda >0\}$. These dilations act on $\mathfrak g$ as 
 	$$
 	(\delta_{\lambda})_{|_{V_i}}:=\lambda^{i}(\mathrm{id})_{|_{V_i}}, \qquad \forall \lambda>0, \quad \forall 1\leq i\leq s,
 	$$
 	and are extended linearly. We will indicate with $\delta_{\lambda}$ both the dilations on $\mathfrak g$ and the group automorphisms corresponding to them via the exponential map. For some features of the general theory of homogeneous Lie groups we refer the reader to \cite[Chapter 1, Section A]{FS82}. For recent developments we refer the reader to \cite{LDNG19}.
 	
	We recall that $\|\cdot\|$ is a {\em homogeneous norm} on a Carnot group $\mathbb G$ if it is continuous from $\mathbb G$ to $[0,+\infty)$ and 
 	\begin{equation}
 	\begin{split}
 	\|g\|&=0, \qquad \mbox{if and only if $g=0$}, \\
 	\|\delta_{\lambda}g\|&=\lambda\|g\|, \qquad \forall \lambda>0, \quad \forall g\in\mathbb G,\\
 	\|g\|&=\|g^{-1}\|, \qquad \forall g\in\mathbb G.
 	\end{split}
 	\end{equation}
 	On a Carnot group a homogeneous norm always exists and moreover two arbitrary homogeneous norms are bi-Lipschitz equivalent. 
 	Moreover if $\|\cdot\|$ is a homogeneous norm on a Carnot group $\mathbb G$ there exists $C>0$ such that $\|xy\|\leq C(\|x\|+\|y\|)$ for every $x,y \in\mathbb G$, see \cite[Proposition 1.6]{FS82}. We can always construct, on an arbitrary Carnot group, a homogeneous norm such that the previous $C$ is $1$: it suffices to take any Carnot-Carathéodory distance from the identity element of $\mathbb G$, see \cite[Section 3.3]{LD17}.
 	
 	On a Carnot group $\mathbb G$ with a stratification $\mathfrak g=V_1\oplus\dots\oplus V_s$, let us set $m_0\coloneqq 0$ and $m_j\coloneqq \dim{V_j}$ for any $j=1,\dots,s$. We stress that $m=m_1$. Let us define $n_0\coloneqq 0$, and $n_j\coloneqq \sum_{\ell=1}^j m_{\ell}$ for any $j=1,\dots,s$. The ordered set $(X_1,\dots, X_n)$ is an {\em adapted basis} for $\mathfrak g$ if the following facts hold. 
 	\begin{itemize}
 		\item[(i)] The vector fields $X_{n_j+1},\dots,X_{n_{j+1}}$ are chosen among the iterated commutators of order $j$ of the vector fields $X_1,\dots, X_m$, for every $j=1,\dots,s-1$.
 		\item[(ii)] The set $\{X_{n_j+1},\dots,X_{n_{j+1}}\}$ is a basis for $V_{j+1}$ for every $j=0,\dots, s-1$.
 	\end{itemize} 
 	If we fix an adapted basis $(X_1,\dots,X_n)$, and $\ell \in \{1,\dots,n\}$, we define the {\em holonomic degree of $\ell$} to be the unique $j^*\in\{1,\dots,s\}$ such that $n_{j^*-1}+1\leq \ell \leq n_{j^*}$. We denote $\deg\ell \coloneqq j^*$ and we also say that $j^*$ is the {\em holonomic degree of $X_{\ell}$}, i.e., $\deg(X_{\ell})\coloneqq j^*$. If an adapted basis $(X_1,\dots,X_n)$ of the Lie algebra $\mathfrak g$ of a Carnot group $\mathbb G$ is fixed, we identify $x\in\mathbb G$ with a point of $\mathbb R^n$ through {\em exponential coordinates of the first kind} as follows
 	$$
 	x\equiv (x_1,\dots,x_n)\leftrightarrow \exp(x_1X_1+\dots+x_nX_n).
 	$$
 	We recall that the {\em homogeneous degree} of the monomial $x_1^{a_1}\cdot\dots\cdot x_n^{a_n}$ in exponential coordinates of the first kind associated to the adapted basis $(X_1,\dots,X_n)$ is $\sum_{\ell=1}^n a_{\ell}\cdot \deg \ell $.
 	
 		We recall here the definition of free-nilpotent Lie algebras see \cite[Definition 14.1.1]{BLU07}.
 	\begin{definition}[Free-nilpotent Lie algebras of step $s$ with $m$ generators]\label{def:free}
 		Let $m\geq 2$ be an integer number. We say that $\mathfrak{f}_{m,s}$ is the {\em free-nilpotent Lie algebra of step $s$ with $m$ generators} $X_1',\dots , X_m'$ if the following facts hold.
 		\begin{itemize}
 			\item[(i)] $\mathfrak{f}_{m,s}$ is a Lie algebra generated by the elements $X_1',\dots , X_m'$, i.e., $\mathfrak{f}_{m,s}$ is the smallest subalgebra of $\mathfrak{f}_{m,s}$ containing $\{X_1',\dots , X_m'\}$;
 			\item[(ii)] $\mathfrak{f}_{m,s}$ is nilpotent of step $s$, i.e., nested Lie brackets of length $s+1$ are always $0$;
 			\item[(iii)] for every nilpotent Lie algebra $\mathfrak g$ of step $s$ and for every map $\Psi \colon \{X_1',\dots , X_m' \}\to \mathfrak g$, there exists a unique homomorphism of Lie algebras $\overline{\Psi}: \mathfrak{f}_{m,s} \to \mathfrak g$ that extends $\Psi$. 
 		\end{itemize}
 		We stress that every free-nilpotent Lie algebra is stratifiable, see \cite[Example 2.5]{LD17}, with $\mathrm{span}\{X_1',\dots,X_m'\}$ being the first layer of a stratification. Thus there exists a unique Carnot group $\mathbb F_{m,s}$ such that its Lie algebra is the free-nilpotent Lie algebra of step $s$ and with $m$ generators.
 	\end{definition}
 	
 	\section{Polynomial and $S$-polynomial distributions on Lie groups}\label{sec:3}
 	In \cref{sec:3.1} we introduce the notion of polynomial distribution with respect to a subset $S$ of the Lie algebra $\mathfrak g$ of an arbitrary Lie group $\mathbb G$. Roughly speaking we say that a distribution $f$ is polynomial with respect to $S$, or briefly $S$-polynomial, when for every $X\in S$ there exists $k$ such that $X^kf\equiv 0$ in the sense of distributions on $\mathbb G$, see \cref{def:kPolynomial}. When $k$ is independent on $X\in S$ we say that $f$ is $k$-polynomial with respect to $S$, or, that is the same, $S$-polynomial with degree at most $k$. We introduce also the definition of polynomial distribution on arbitrary Lie groups $\mathbb G$, see \cref{def:PolyGeneral}: namely, a distribution on $\mathbb G$ is polynomial if there exists $k\in\mathbb N$ such that for all $X_1,\dots,X_k\in\mathfrak g$ we have $X_1\dots X_kf\equiv 0$ in the sense of distributions on $\mathbb G$. This latter definition happens to be consistent with the definition of polynomial map between groups introduced by Leibman in \cite{Lei02}, see \cref{rem:LeiBravo}. We conclude the subsection by proving a lemma about the pointwise limit of smooth $k$-polynomial functions with respect to $X$, where $X\in \mathfrak g$, see \cref{lem:APPROXKPol}, and \cref{coroll:KpolAtLimit}. These latter results about convergence will come into play in the proof of \cref{thm:MAINCARNOT}.
 		
 	In \cref{sec:3.2} we prove formula \eqref{eqn:RepresentationImportant}, see \cref{lem:Representation}. Namely, given a smooth function $f:\mathbb G\to\mathbb R$ for which $X^kf\equiv 0$ for some $X\in \mathfrak g$ and $k\in\mathbb N$, we show that knowing $f$ on an open set $U\subseteq \mathbb G$ completely determines $f$ on the set $U\cdot\exp(\mathbb R X)$. Then we use \cref{lem:Representation} to prove the fundamental representation formula \eqref{eqn:YEAHBISAdapted} in \cref{lem:EST1BISAdapted} according to which if $f$ is $S$-polynomial with degree at most $k$, with a Lie generating $S$, then it is completely determined by the jet of some sufficiently big order of $f$ at the identity.
 	
 	In \cref{sec:3.3} we prove \cref{thm:Intro1}.
 	The proof of \cref{thm:Intro1}, see \cref{prop:FiniteDim} and \cref{proofThm1}, is reached by means of \cref{lem:ConnectionBIS}, according to which in a connected Lie group $\mathbb G$ around every point there exists a chart that is a concatenation of horizontal curves, and by means of the representation formula in \cref{lem:EST1BISAdapted}. 
 		
 		Before starting the discussion, let us recall some basic facts about the analytic structure of a Lie group. It is a classical result of Gleason, Montgomery and Zippin that a topological group that has the structure of a $C^k$-manifold for some $0\leq k\leq +\infty$ admits exactly one analytic structure that is compatible with the $C^k$-structure, see the discussion in \cite[page 42]{Varadarajan}. Thus, on a Lie group, we can give a meaning to a function $f:\mathbb G\to\mathbb R$ being analytic by using an analytic atlas.
 	
 	\subsection{Relations with pointwise convergence}\label{sec:3.1} In what follows we give the definitions of $S$-polynomial and polynomial distributions on Lie groups.
 	\begin{definition}[$S$-polynomial distributions on Lie groups]\label{def:kPolynomial}
 		Let $\mathbb G$ be a Lie group with Lie algebra $\mathfrak g$, and let us fix a subset $S\subseteq \mathfrak g$. We say that a distribution $f\in\mathcal{D}'(\mathbb G)$ is {\em polynomial with respect to $S$}, or {\em horizontally polynomial} if $S$ is understood, if 
 		$$
 		\forall X\in S\,\, \exists k\in\mathbb N \,\,\text{such that}\,\, X^k f\equiv 0\,\, \text{holds on $\mathbb G$ in the sense of distributions.}
 		$$ 
 		If the previous condition holds, we also say that $f$ is {\em $S$-polynomial}. If the choice of $k$ is uniform on $X\in S$, we say that $f$ is {\em $k$-polynomial with respect to $S$}, or {\em horizontally $k$-polynomial} if $S$ is understood. If the previous condition holds, we also say that $f$ is {\em $S$-polynomial with degree at most $k$}. When $k=2$ we say that $f$ is {\em affine with respect to $S$}, or {\em horizontally affine} if $S$ is understood.
 	\end{definition}
 \begin{remark}[Taylor expansion for $S$-polynomial smooth functions]\label{rem:Taylor}
 	It is easy to notice that if $f:\mathbb G\to\mathbb R$ is a smooth function 
 	that is $k$-polynomial with respect to $X$, then 
 	\begin{equation}\label{eqn:taylor}
 	f(p\exp(tX))=f(p)+t(Xf)(p)+\dots+\frac{t^{k-1}}{(k-1)!}(X^{k-1}f)(p), \quad \forall p\in\mathbb G, \forall t\in\mathbb R.
 	\end{equation}
 	The previous observation comes from the fact that for an arbitrary smooth function $f:\mathbb G\to \mathbb R$ the following equality holds
 	$$
 	X^jf(p)=\frac{\de^j}{\de\eps^j}_{|_{\eps=0}} f(p\exp(\eps X)), \quad \forall p\in\mathbb G,\forall j\in\mathbb N.
 	$$
 	Moreover, let us notice that if $f:\mathbb G\to \mathbb R$ is a function such that it is a polynomial with degree at most $k$ along the flow lines of $X\in\mathfrak{g}$, then $X^kf\equiv 0$ holds on $\mathbb G$ in the classical sense and we can write the expansion in \eqref{eqn:taylor}. To be more precise, if we fix $p\in\mathbb G$ and there exist $k$ real numbers $a_0,\dots,a_{k-1}$ such that 
 	\begin{equation}\label{eqn:FromPolToXk}
 	f(p\exp(tX))=a_0+ta_1+\dots+t^{k-1}a_{k-1}, \quad \forall t\in\mathbb R,
 	\end{equation}
 	then $f$ is differentiable $k$ times along $X$ at $p$, $X^jf(p)=j!a_j$ for every $j=0,\dots,k-1$, and $X^kf(p)=0$.
 \end{remark}
 \begin{definition}(Polynomial distributions on Lie groups)\label{def:PolyGeneral}
 	A distribution $f\in\mathcal{D}'(\mathbb G)$ is a {\em polynomial distribution} if there exists a $k_0\in\mathbb N$ such that for every $Y_1,\dots,Y_{k_0}\in\mathfrak g$ we have $Y_1\dots Y_{k_0}f \equiv 0$ on $\mathbb G$.
 \end{definition}
 \begin{remark}
 	Due to \cref{thm:Intro1}, one can equivalently ask $f$ to be an analytic function in \cref{def:PolyGeneral}.
 \end{remark}
 \begin{remark}[Comparison between \cref{def:PolyGeneral} and Leibman's definition in \cite{Lei02}]\label{rem:LeiBravo}
 	In \cite{Lei02} the author gives and studies the notion of polynomial map $f:\mathbb G\to \mathbb H$ between two groups $\mathbb G$ and $\mathbb H$. Given $g\in \mathbb G$ we define the operator $D_g$ that acts on functions $f:\mathbb G\to \mathbb H$ as follows
 	\begin{equation}\label{eqn:DefinitionDg}
 	(D_gf)(g'):=f(g')^{-1}f(g'g).
 	\end{equation}
 	According to \cite[Section 0.2]{Lei02}, a map $f:\mathbb G\to \mathbb H$ between two groups $\mathbb G$ and $\mathbb H$ is a {\em polynomial map with degree at most $d$}, being $d\in\mathbb N$, if for every $g_1,\dots,g_{d+1}\in \mathbb G$ we have 
 	\begin{equation}\label{eqn:gi0}
 	D_{g_1}\dots D_{g_{d+1}} f \equiv e_{\mathbb H},
 	\end{equation}
 	where $e_{\mathbb H}$ is the (function that is constantly equal to the) identity of $\mathbb H$. In our case, i.e., when $\mathbb H = (\mathbb R,+)$, we can give a definition that mimics the previous one but for distributions $f$ on $\mathbb G$. Let us define the operator $D_g$ acting on distributions $f$ as in \eqref{eqn:ActionDgDistr}, thus generalizing \eqref{eqn:DefinitionDg} for distributions.
 	We say that a distribution $f$ on $\mathbb G$ is {\em polynomial à la Leibman with degree at most $d\in\mathbb N$} if for every $g_1,\dots,g_{d+1}\in \mathbb G$ we have
 	\begin{equation}\label{eqn:gi}
 	D_{g_1}\dots D_{g_{d+1}} f\equiv 0, \qquad \text{in the sense of distributions on $\mathbb G$}.
 	\end{equation}
 	Let us notice that if $f$ is continuous the two definitions in \eqref{eqn:gi0} and \eqref{eqn:gi} agree. However there are non-continuous functions, already from $\mathbb R$ to $\mathbb R$, that satisfy \eqref{eqn:DefinitionDg} but they cannot be seen as distributions, see \cref{rem:BLUeKP}.
 	
 	We stress that the result in \cref{prop:Equ1Intro} tells us that \cref{def:PolyGeneral} and Leibman's definition adapted for distributions give raise to the same class of distributions.
 \end{remark}

We prove the following lemma about the pointwise limit of functions that are polynomials along one line in the direction of $X\in\mathfrak g$ emanating from a fixed point $p\in\mathbb G$. We are going to prove that the pointwise limit of such functions, whenever it exists, is still polynomial along the same line. 
\begin{lemma}\label{lem:APPROXKPol}
	Let $\mathbb G$ be a Lie group. Let $p\in\Omega\subseteq\mathbb G$, where $\Omega$ is open, let $X$ be a left-invariant vector field on $\mathbb G$, and let $k\in\mathbb N$. Let $0\in I\subseteq \mathbb R$ be an interval such that $p\exp(tX)\in\Omega$ for all $t\in I$. Let $\{f_n\}_{n\in\mathbb N}$ be a sequence of functions $f_n:\Omega \to\mathbb R$ such that, for every $n\in\mathbb N$, there exists a sequence of $k$-uples of real numbers $\{(a_{n,0},\dots,a_{n,k-1})\}_{n\in\mathbb N}$ such that
	$$
	f_n(p\exp(tX))=a_{n,0}+ta_{n,1}+\dots+t^{k-1}a_{n,k-1}, \qquad \forall t\in I.
	$$
	Let us assume that there exists a function $f:\Omega\to \mathbb R$ such that $f_n\to f$ pointwise on the set $\{p\exp(tX):t\in I\}$, as $n\to +\infty$. Then there exists a $k$-uple of real numbers $(a_0,\dots,a_{k-1})$ such that $a_{n,i}\to a_i$ for all $i=0,\dots, k-1$ and for $n\to+\infty$, and 
	$$
	f(p\exp(tX))=a_0+ta_1+\dots+t^{k-1}a_{k-1}, \qquad \forall t\in I.
	$$
\end{lemma}
\begin{proof}
	Let us fix $k$ pairwise distinct nonzero real numbers in $I$ and let us call them $r_0,\dots, r_{k-1}$. By hypothesis we get that, for every $n\in\mathbb N$, the following holds
	$$
	\sum_{i=0}^{k-1} r_j^i a_{n,i}\xrightarrow{n\to+\infty} f(p\exp(r_jX)), \qquad \forall j=0,\dots,k-1.
	$$
	For every $n\in\mathbb N$ we denote with $\overline a_n$ the column $k$-vector $(a_{n,0},\dots,a_{n,k-1})$, and with $\overline f_{p,X}$ the column $k$-vector $(f(p\exp(r_jX)))_{j=0,\dots,k-1}$. Thus we can write the previous convergence as follows
	$$
	V\cdot \overline a_n\xrightarrow{n\to+\infty} \overline f_{p,X},
	$$
	where $V$ is the $k\times k$ Vandermonde's matrix $V_{ji}:= (r_j^i)_{j,i=0,\dots,k-1}$. Since $V$ is invertible, we conclude that 
	$$
	\overline a_n\xrightarrow{n\to+\infty} V^{-1}\cdot \overline f_{p,X},
	$$
	and thus, if we denote $(a_0,\dots,a_{k-1})$ the components of the vector $V^{-1}\cdot \overline f_{p,X}$, we conclude that $a_{n,i}\to a_i$ for all $i=0,\dots,k-1$ and for $n\to+\infty$. The last part of the statement easily follows from the latter convergence and the fact that $f_n\to f$ pointwise on $\{p\exp(tX):t\in I\}$, as $n\to+\infty$. 
\end{proof}
\begin{corollary}\label{coroll:KpolAtLimit}
	Let $\mathbb G$ be a Lie group, $\Omega\subseteq \mathbb G$ be open, $X$ be a left-invariant vector field on $\mathbb G$, and $k\in\mathbb N$. Let $\{f_n\}_{n\in\mathbb N}$ be a sequence of smooth functions $f_n:\Omega\to \mathbb R$ such that $X^kf_n\equiv 0$ on $\Omega$ for every $n\in\mathbb N$. If there exists a function $f:\Omega\to\mathbb R$ such that $f_n\to f$ pointwise on $\Omega$, then $X^kf\equiv 0$ on $\Omega$ in the classical sense.
\end{corollary}
\begin{proof}
	If we fix $p\in \Omega$, since $\Omega$ is open there exists an interval $0\in I\subseteq \mathbb R$ such that $\{p\exp(tX):t\in I\}\subseteq\Omega$. From \cref{rem:Taylor}, see in particular the localized version of \eqref{eqn:taylor}, we get that the sequence $\{f_n\}_{n\in\mathbb N}$ satisfies the hypotheses of \cref{lem:APPROXKPol}. Thus, applying \cref{lem:APPROXKPol}, the function $f$ coincides with a polynomial with degree at most $k$ in $t$ along the piece of line $\{p\exp(tX):t\in I\}$, and arguing as at the end of \cref{rem:Taylor}, see in particular the localized reasoning above and below \eqref{eqn:FromPolToXk}, we get that $X^kf(p)=0$ and then we get the conclusion since $p\in\Omega$ is arbitrary.
\end{proof}
\subsection{Propagation of being $S$-polynomial}\label{sec:3.2} In the following lemma we are going to prove a formula that will be of crucial importance in the proof of \cref{thm:Intro1} and \cref{thm:Intro2}, since it is the main tool to prove the representation formulas in \cref{lem:EST1BISAdapted}, and \cref{lem:EST1BIS}.
\begin{lemma}\label{lem:Representation}
	Let $\mathbb G$ be a Lie group. Let us fix $k\in\mathbb N$, $X\in\mathfrak{g}$, and let $f:\mathbb G\to \mathbb R$ be a smooth function such that $X^kf\equiv 0$ on $\mathbb G$. Then if we fix $r\in\{0,1,2,\dots\}$ and $X_1,\dots,X_r\in \mathfrak g$, the following formula holds
	\begin{equation}\label{eqn:RepresentationImportant}
	(X_1\ldots X_r f)(q\exp(tX)) = \sum_{i=0}^{k-1} \frac{t^i}{i!} (\mathrm{Ad}_{\exp(tX)}(X_1)\ldots\mathrm{Ad}_{\exp(tX)}(X_r)X^i f)(q),
	\end{equation}
	for every $q\in \mathbb G$ and every $t\in\mathbb R$.
\end{lemma}
\begin{proof}
	Let us prove the statement by induction on $r$. If $r=0$ the formula holds true by \eqref{eqn:taylor}. Let us now suppose that the statement holds true for some $r\in \{0,1,2,\dots\}$ and let us prove that it holds true for $r+1$. Indeed, let us fix $X_1,\dots,X_{r+1}\in \mathfrak g$, $t\in \mathbb R$, and $q\in\mathbb G$. Then we compute the derivative of $X_2\dots X_{r+1}f$ along $X_1$ as follows
	\begin{equation}\label{eqn:Prev}
	\begin{split}
	X_1(X_2\ldots X_{r+1}f)&(q\exp(tX))=\frac{\de}{\de\eps}_{|_{\eps=0}}(X_2\ldots X_{r+1}f)(q\exp(tX)\exp(\eps X_1)) \\
	&=\frac{\de}{\de\eps}_{|_{\eps=0}}(X_2\ldots X_{r+1}f)(\underbrace{q\exp(tX)\exp(\eps X_1)\exp(-tX)}_{q_\eps}\exp(tX))\\
	&=\frac{\de}{\de\eps}_{|_{\eps=0}}\sum_{i=0}^{k-1} \frac{t^i}{i!} \mathrm{Ad}_{\exp(tX)}(X_2)\ldots\mathrm{Ad}_{\exp(tX)}(X_{r+1})X^i f(q_\eps) \\
	&=\sum_{i=0}^{k-1} \frac{t^i}{i!}\frac{\de}{\de\eps}_{|_{\eps=0}}\mathrm{Ad}_{\exp(tX)}(X_2)\ldots\mathrm{Ad}_{\exp(tX)}(X_{r+1})X^i f(q_\eps) \\
	&=\sum_{i=0}^{k-1} \frac{t^i}{i!}\mathrm{Ad}_{\exp(tX)}(X_1)\ldots\mathrm{Ad}_{\exp(tX)}(X_{r+1})X^i f(q),
	\end{split}
	\end{equation}
	where in the third equality we used the inductive hypothesis, and in the last equality we used that the curve $\eps\to q_{\eps}$ has $\mathrm{Ad}_{\exp(tX)}(X_1)_{|_q}$ as tangent vector at $\varepsilon=0$.
\end{proof}
	\begin{lemma}\label{lem:EST1BISAdapted}
		Let $\mathbb G$ be a Lie group, and let $f:\mathbb G\to\mathbb R$ be smooth and $k$-polynomial with respect to $\{Y_1,\dots,Y_\ell\}\subseteq \mathfrak g$ for some $k,\ell \in\mathbb N$. Given $t_1,\dots,t_\ell\in\mathbb R$, let us define, for every $j=2,\dots,\ell$, the element of the group $g_j:=\exp(t_1Y_1)\dots\exp(t_{j-1}Y_{j-1})$, and let us denote $g_1:=e$ the identity of the group. Then the following equality holds
		\begin{equation}\label{eqn:YEAHBISAdapted}
			f(\exp(t_1Y_1)\dots\exp(t_\ell Y_\ell))= \sum_{i_1,\dots,i_{\ell}=0}^{k-1} \frac{t_1^{i_1}\dots t_\ell^{i_\ell}}{i_1!\dots i_\ell!}\left((\mathrm{Ad}_{g_\ell}Y_\ell)^{i_\ell}(\mathrm{Ad}_{g_{\ell-1}}Y_{\ell-1})^{i_{\ell-1}}\dots(\mathrm{Ad}_{g_1}Y_1)^{i_1} f\right)(e).
		\end{equation}
	\end{lemma}
	\begin{proof}
		Let us show, for simplicity, the computations only in the nontrivial case $\ell=3$, while the general case follows along the same lines and we omit it. For every $(t_1,t_2,t_3)\in\mathbb R^3$ the following equality holds
		\begin{equation*}
			\begin{split}
				f(\exp(t_1Y_1)&\exp(t_2Y_2)\exp(t_3Y_3))=\sum_{i_3=0}^{k-1}\frac{t_3^{i_3}}{i_3!} (Y_3^{i_3}f)(\exp(t_1Y_1)\exp(t_2Y_2)) \\
				&=\sum_{i_3=0}^{k-1}\frac{t_3^{i_3}}{i_3!}\sum_{i_2=0}^{k-1} \frac{t_2^{i_2}}{i_2!}\left((\mathrm{Ad}_{\exp(t_2Y_2)}Y_3)^{i_3}Y_2^{i_2}f\right)(\exp(t_1Y_1)) \\
				&=\sum_{i_3=0}^{k-1}\frac{t_3^{i_3}}{i_3!}\sum_{i_2=0}^{k-1} \frac{t_2^{i_2}}{i_2!}\sum_{i_1=0}^{k-1}\frac{t_1^{i_1}}{i_1!}\left((\mathrm{Ad}_{\exp(t_1Y_1)}\mathrm{Ad}_{\exp(t_2Y_2)}Y_3)^{i_3}(\mathrm{Ad}_{\exp(t_1Y_1)}Y_2)^{i_2}Y_1^{i_1}f\right)(e) \\
				&=\sum_{i_1,i_2,i_3=0}^{k-1}\frac{t_1^{i_1}t_2^{i_2}t_3^{i_3}}{i_1!i_2!i_3!}\left((\mathrm{Ad}_{\exp(t_1Y_1)\exp(t_2Y_2)}Y_3)^{i_3}(\mathrm{Ad}_{\exp(t_1Y_1)}Y_2)^{i_2}Y_1^{i_1}f\right)(e),
			\end{split}
		\end{equation*}
		where in the first equality we used that $f$ is $k$-polynomial with respect to $Y_3$; in the second equality we used \eqref{eqn:RepresentationImportant} with $r=i_3$, $X_1=\dots=X_r=Y_3$, $X=Y_2$, and $q=\exp(t_1Y_1)$; in the third equality we used again \eqref{eqn:RepresentationImportant} with $r=i_3+i_2$, $X_1=\dots=X_{i_3}=\mathrm{Ad}_{\exp(t_2Y_2)}Y_3$, $X_{i_3+1}=\dots=X_{i_3+i_2}=Y_2$, $X=Y_1$, and $q=e$; and in the fourth equality we used that $\mathrm{Ad}_g\mathrm{Ad}_h=\mathrm{Ad}_{gh}$ for every $g,h\in\mathbb G$.
\end{proof}
\subsection{Proof of \cref{thm:Intro1}}\label{sec:3.3}
Before starting the proof of \cref{thm:Intro1} we recall here a lemma that tells us that on a connected Lie group we can find local charts by concatenating a fixed amount of flow lines of horizontal vector fields. This is a standard result in control theory.
\begin{lemma}[{\cite[Lemma 3.33]{ABB19}}]\label{lem:ConnectionBIS}
	Let $\mathbb G$ be a connected Lie group of topological dimension $n$ and let $S$ be a subset of the Lie algebra $\mathfrak g$ that Lie generates $\mathfrak g$. 
	Then there exists an open neighbourhood $U$ of the identity $e$ and $2n$ elements $X_{1},\dots,X_{2n}\in S$ such that for all $p\in \mathbb G$ and all $q\in p\cdot U$ there exist $s_1,\dots,s_{2n}\in\mathbb R$  such that 
	$$
	q=p\exp(s_1X_{1})\dots \exp(s_{2n} X_{2n}); 
	$$
	more precisely there exist an open bounded set $\hat V\subseteq (0,1)^n$ and $(\hat s_1,\dots,\hat s_n)\in (0,1)^n$ such that 
	\begin{equation}\label{eqn:HatPsi}
	\begin{split}
	\hat\psi_p\colon\hat V&\to p\cdot U, \\ \hat\psi_p(s_1,\dots,s_n)\coloneqq p\exp(s_1X_{1})\dots&\exp(s_nX_{n})\exp(-\hat s_nX_{n+1})\dots\exp(-\hat s_1X_{2n}),
	\end{split}
	\end{equation}
	is a diffeomorphism for every $p\in\mathbb G$.
\end{lemma}
As a first step toward the proof of \cref{thm:Intro1}, we prove the finite-dimensional result in the second part of \cref{thm:Intro1} for analytic functions on connected Lie groups. The proof of the forthcoming \cref{prop:FiniteDim} is reached by joining the previous representation formula proved in \cref{lem:EST1BISAdapted} with \cref{lem:ConnectionBIS}.
\begin{proposition}\label{prop:FiniteDim}
	Let $\mathbb G$ be a connected Lie group of topological dimension $n$, and let $S$ be a Lie generating subset of $\mathfrak g$. Then for every $k\in\mathbb N$ there exists $\delta$, which depends on $k$ and $n$, such that the vector space 
	$$
	\mathscr F:=\{f:\mathbb G\to\mathbb R\,\text{is analytic and $k$-polynomial with respect to $S$} \},
	$$
	is finite-dimensional and its dimension is bounded above by $\delta$.
\end{proposition}
\begin{proof}
	Let $\hat\psi_e:\hat V\subseteq \mathbb R^n\to U\subseteq\mathbb G$ be the local chart around $e$ constructed as in \eqref{eqn:HatPsi}. This chart induces a local frame $(\partial_{x_1},\dots,\partial_{x_n})$ of the tangent bundle $TU$. We denote with $\alpha:=(\alpha_1,\dots,\alpha_n)$ an arbitrary $n$-uple of natural numbers, we set $\lvert \alpha \rvert := \alpha_1+\dots+\alpha_n$ and we denote
	$$
	\partial_{x^\alpha}:=\partial_{\underbrace{x_1\dots x_1}_{\alpha_1}\cdots \underbrace{x_n\dots x_n}_{\alpha_n} }.
	$$
	Recall that $\hat\psi_e$ has the following explicit expression
	$$
	\hat\psi_e(s_1,\dots,s_n)= \exp(s_1X_{1})\dots\exp(s_nX_{n})\exp(-\hat s_nX_{n+1})\dots\exp(-\hat s_1X_{2n}),
	$$
	for some $X_1,\dots,X_{2n}\in S$, and $(\hat s_1,\dots,\hat s_n)\subseteq(0,1)^n$. Moreover, since $\mathrm{Ad}$, the exponential map, and the product on the group $\mathbb G$ are analytic functions, from \eqref{eqn:YEAHBISAdapted} read in coordinates and the previous explicit expression of $\hat\psi_e$ we get that there exists $D:=D(k,n)$ such that 
	\begin{equation}\label{eqn:Cidiceiniettiva}
		f(\hat\psi_e(s_1,\dots,s_n))= \sum_{\lvert \alpha \rvert \leq D} h_{\alpha}(s_1,\dots,s_n)(\partial_{x^{\alpha}}f)(e),
	\end{equation}
	for every $(s_1,\dots,s_n)\in \hat V$, and where, for every multi-index $\alpha$, the function $h_\alpha(s_1,\dots,s_n):\hat V \to\mathbb R$ is an analytic function depending on the chart $\hat\psi_e$. Now, there exists $\delta:=\delta(D)$ such that the number of the operators $(\partial_{x^{\alpha}})_{|_e}$ with $|\alpha|\leq D$ is $\delta$.	Let us define the vector space
	$$
	\mathscr F':=\{f_{|_U},\,\text{where $f:\mathbb G\to\mathbb R$ is analytic and $k$-polynomial with respect to $S$} \},
	$$
	and let $\Psi:\mathscr F'\to\mathbb R^{\delta}$ be the linear map defined by $\Psi(f_{|_U}):=((\partial_{x^{\alpha}}f)(e))_{|\alpha|\leq D}$ for every $f_{|_U}\in \mathscr F'$. From \eqref{eqn:Cidiceiniettiva} we get that $\Psi$ is an injective map, and thus the dimension of $\mathscr F'$ is finite and bounded above by $\delta$. 
	
	Let $\Psi':\mathscr{F}\to\mathscr F'$ be the linear map defined by $\Psi'(f)=f_{|_U}$ for every analytic function $f$ on $\mathbb G$. By analytic continuation we deduce that whenever $f_{|_U}\equiv g_{|_U}$ for two analytic functions, then $f\equiv g$ on $\mathbb G$.
	Thus $\Psi'$ is injective as well and we deduce that $\mathscr F$ is finite-dimensional and its dimension is bounded above by $\delta$.
\end{proof}
\begin{proof}[Proof of \cref{thm:Intro1}]\label{proofThm1}
	We first notice that we can reduce to work with the case $\mathbb G$ is connected. Indeed, the connected component of the identity of $\mathbb G$ is itself a Lie group, and if we know the result for such a connected component we obtain the result for all the connected components by composing the distribution $f$ to the right with some left translation, which preserves the condition of being $S$-polynomial. Thus, from now on in the proof, we assume $\mathbb G$ is connected.
	
	Up to taking a subset of $S$ that is finite and still Lie generates $\mathfrak g$, we may assume that $S$ is finite. Since now $S$ is finite, there exists $k\in\mathbb N$ such that $f$ is $k$-polynomial with respect to $S$. 
	Let $\phi_n$ be an approximate identity as in \cref{def:approximateidentity}. Thus $\phi_n \ast f$ is a smooth function on $\mathbb G$, see \cref{rem:DerivativeAndConvolution}, and $X(\phi_n\ast f)=\phi_n\ast Xf$ for every $X\in\mathfrak g$, see \eqref{rem:DerivativeOfAConvolution}. Thus, iterating \eqref{rem:DerivativeOfAConvolution}, $\phi_n\ast f$ is smooth and $k$-polynomial with respect to $S$. We claim that for every $n\in\mathbb N$ the function $\phi_n\ast f$  is analytic on $\mathbb G$.
	
	Indeed, for the sake of notation, let us fix $n\in\mathbb N$ and let us rename $h:=\phi_n \ast f$. For every $g\in\mathbb G$, the function $h\circ L_g$, where $L_g$ is the left translation by $g\in\mathbb G$, is still smooth and $k$-polynomial with respect to $S$. Thus in order to prove that $h$ is analytic it is sufficient to prove that $h\circ L_g$ is analytic in a neighbourhood of $e$ for every $g\in\mathbb G$. In conclusion, in order to prove that $h$ is analytic on $\mathbb G$, we only need to prove that every $\widetilde h\in C^{\infty}(\mathbb G)$ which is $k$-polynomial with respect to $S$ is analytic in a neighbourhood of the identity. Let $\hat\psi_e:\hat V\subseteq \mathbb R^n\to U\subseteq\mathbb G$ be the local chart around $e$ constructed as in \eqref{eqn:HatPsi} from $S$. From the fact that for every $(s_1,\dots,s_n)\in\hat V$ we have
	$$
	\hat\psi_e(s_1,\dots,s_n)= \exp(s_1X_{1})\dots\exp(s_nX_{n})\exp(-\hat s_nX_{n+1})\dots\exp(-\hat s_1X_{2n}),
	$$
	for some $X_1,\dots,X_{2n}\in S$ and $(\hat s_1,\dots,\hat s_n)\subseteq\mathbb R^n$, and the fact that the representation formula \eqref{eqn:Cidiceiniettiva} holds with some analytic functions $h_{\alpha}$ that only depend on the chart $\hat\psi_e$, we conclude that, for every $\widetilde h\in C^{\infty}(\mathbb G)$ that is $k$-polynomial with respect to $S$, the function
	$\widetilde{h}\circ\hat\psi_e$ is an analytic real-valued function defined on $\hat V$. Moreover $\hat\psi_e$ is an analytic map, since it is a composition of the exponential map with the product of the group. Since $\hat\psi_e$ is invertible being a chart, by the inverse function theorem we conclude that $(\hat\psi_e)^{-1}$ is analytic as well\footnote{This assertion follows from the fact that if a diffeomorphism $\xi$ between open subsets of $\mathbb R^n$ is analytic, then $\xi^{-1}$ is analytic as well. Indeed, there exists a complex holomorphic extension of $\xi$, say $\widetilde \xi$, between open subsets of $\mathbb C^n$. At every $x\in \mathrm{dom}(\xi)$, the Jacobian determinant $\det ((J\xi)_x)$ is nonzero since $\xi$ is a diffeomorphism. Then, when seeing $x$ as an element of $\mathbb C^n$, the complex Jacobian determinant $\det ((J\widetilde\xi)_x)$, see \cite[page 30]{Hol}, is nonzero as well due to \cite[Chapter I, Theorem 7.2]{Hol}. Thus from the complex Inverse Function Theorem, see \cite[Chapter I, Theorem 7.5]{Hol}, $\widetilde\xi$ is biholomorphic in a neighbourhood of $x$ seen in $\mathbb C^n$, and then, by restriction, $\xi^{-1}$ is analytic in a neighbourhood of $\xi(x)$.} as a map from $U$ to $\hat V$. Thus, for every $\widetilde h\in C^{\infty}(\mathbb G)$ that is $k$-polynomial with respect to $S$, we have that $\widetilde h=(\widetilde h\circ \hat\psi_e)\circ(\hat\psi_e)^{-1}$ is analytic as a map from $U$ to $\mathbb R$, since it is the composition of analytic functions, and then the proof of the claim is concluded.  
	
	Hence $\phi_n \ast f\in\mathscr F$, see \cref{prop:FiniteDim}, and $\mathscr F$ is closed in the weak*-topology of $\mathcal{D}'(\mathbb G)$ since it is finite-dimensional, see \cite[Theorem 1.21]{Rudin}. Since $\phi_n\ast f\to f$ in the topology of $\mathcal{D}'(\mathbb G)$, because $\phi_n$ is an approximate identity, we conclude that $f$ has a representative in $\mathscr F$ as well, and thus $f$ is represented by an analytic function. 
	
	In order to prove the second part of \cref{thm:Intro1}, let us fix a distribution $f$ that is $S$-polynomial with degree at most $k\in\mathbb N$. Hence $f$ is represented by an analytic function from what we proved above, and the application of \cref{prop:FiniteDim} concludes the proof.
\end{proof}

\section{The case of nilpotent Lie groups}\label{sec:4}
	{In this section we focus our attention on the case when $\mathbb G$ is a connected nilpotent Lie group. We recall that the exponential map $\exp\colon\mathfrak g\to \mathbb G$ is a global analytic diffeomorphism whenever $\mathbb G$ is a simply connected nilpotent Lie group, while, when $\mathbb G$ is connected and nilpotent but not necessarily simply connected, it is an analytic and surjective map, see \cite[Theorem 3.6.1]{Varadarajan}. Hence, we can use the exponential map to give a natural definition of polynomial in exponential chart. We say that a function $f\colon\mathbb G\to\mathbb R$ is polynomial in exponential chart if $f\circ\exp$ is a polynomial, see \cref{def:Polynomial}. Such a notion, in the nilpotent setting, is consistent, see \cref{rem:EquivalencePolyInSimplyConnected}, with both the definition of polynomial distribution on a Lie group, see \cref{def:PolyGeneral}, and with Leibman's definition in \cref{rem:LeiBravo}. 
		
	In \cref{lem:EST1BIS} we use the formula in \cref{lem:Representation} to prove that whenever a smooth function $f$ is $k$-polynomial with respect to $S$ in a nilpotent Lie group of nilpotency step $s$, then $f$ is polynomial along the concatenation of flows of elements of $S$ emanating from a fixed $p\in\mathbb G$. More precisely we prove that, if $p\in\mathbb G$ is fixed, $\ell\in\mathbb N$, and $Y_1,\dots,Y_\ell\in S$, the map  $(t_1,\dots,t_\ell)\mapsto f(p\exp(t_1Y_1)\dots\exp(t_\ell Y_\ell))$ is a polynomial in the variables $t_1,\dots,t_\ell$ and we explicitly bound the degree of the polynomial with a constant $\nu:=\nu(k,s,\ell)$.
		
	Thus we use the latter result to give the proof of \cref{thm:Intro2} in \cref{sec:4.1}. In order to do so we first prove that on a simply connected nilpotent Lie group $\mathbb G$ a distribution that is polynomial with respect to a set $S$ that Lie generates $\mathfrak g$ is actually polynomial in exponential chart, and then we reduce to the simply connected case by passing to the universal cover. In order to prove the result for $\mathbb G$ nilpotent and simply connected, we reduce ourselves to the case of Carnot groups, by lifting the problem to a free-nilpotent Lie algebra, see the proof of \cref{thm:MAINGENERAL}. \cref{thm:Intro2} in the case of Carnot groups is proved in \cref{thm:MAINCARNOT} and the proof goes as follows. First, we use the main result in \cref{thm:Intro1} to obtain that a distribution that is $k$-polynomial with respect to a set $S$ that Lie generates the Lie algebra is represented by an analytic function. Second, we prove that each homogeneous term in the Taylor series of $f$ around the identity, see \eqref{eqn:SERIES}, is $k$-polynomial with respect to $S$ as well. Third, we conclude because, thanks to the fundamental \cref{lem:EST1BIS}, and thanks to an improvement of \cref{lem:ConnectionBIS} in the setting of Carnot groups, namely \cref{lem:Connecting}, a smooth $k$-polynomial function with respect to $S$ has polynomial growth of bounded order at infinity. Thus the Taylor expansion of $f$ at the identity is finite, and from the fact that $f$ is analytic we conclude that $f$ coincides with this finite Taylor expansion, namely $f$ is a polynomial.
		
	
 We give the following notion of polynomial in exponential chart. Let us stress that the following definition agrees with the one given in \cite[Definition 20.1.1]{BLU07} and therein studied in the more restrictive setting of stratified groups. We also stress that the forthcoming definition does not depend on the choice of a basis of the Lie algebra $\mathfrak g$.
 \begin{definition}[Polynomial in exponential chart on connected nilpotent Lie groups]\label{def:Polynomial}
 	Given a connected nilpotent Lie group $\mathbb G$ of dimension $n$ and a basis $\{X_1,\dots,X_n\}$ of the Lie algebra $\mathfrak g$, we say that $f:\mathbb G\to\mathbb R$ is {\em polynomial in exponential chart} if 
 	$$
 	\mathbb R^n \ni (t_1,\dots,t_n)\mapsto f(\exp(t_1X_1+\dots+t_nX_n)),
 	$$
 	is a polynomial function of the variables $t_1,\dots,t_n$.
 \end{definition}
\begin{remark}[\cref{def:Polynomial}, \cref{def:PolyGeneral}, and Leibman's definition are consistent]\label{rem:EquivalencePolyInSimplyConnected}
	When $\mathbb G$ is a connected and nilpotent Lie group, the result in \cref{prop:Equ2Intro} tells us that the definition of polynomial in exponential chart, see \cref{def:Polynomial}, the definition of polynomial distribution à la Leibman, see \cref{rem:LeiBravo}, and the definition of polynomial distribution, see \cref{def:PolyGeneral}, are equivalent.
\end{remark}
	
In the forthcoming \cref{lem:EST1BIS} we prove that on an arbitrary nilpotent group $\mathbb G$ a smooth $k$-polynomial function with respect to $S$ is polynomial along the concatenation of lines in the directions of $S$. In order to prove this we exploit the formula in \cref{lem:Representation} and the following remark.
\begin{remark}[Formula \eqref{eqn:RepresentationImportant} on nilpotent groups]\label{rem:Nilpoten}
	Let us notice that if $\mathbb G$ is a nilpotent Lie group of nilpotency step $s$ we have that $(\mathrm{ad}_X)^s(Y)\equiv 0$ for all $X,Y\in \mathfrak g$. Then, from \eqref{eqn:AdjointExponential} we get
	$$
	\mathrm{Ad}_{\exp(tX)}(Y)=\sum_{j=0}^{s-1}\frac{1}{j!}(\mathrm{ad}_{tX})^j(Y),
	$$ 
	for very $X,Y\in \mathfrak g$ and every $t\in\mathbb R$. Thus, if we fix $t\in\mathbb R$, $r\in \{0,1,2,\dots\}$, and $X,X_1,\dots,X_r\in\mathfrak g$, we conclude that $\mathrm{Ad}_{\exp(tX)}(X_1)\ldots\mathrm{Ad}_{\exp(tX)}(X_r)$ is a sum of left-invariant operators of the form $X_{i_1}\dots X_{i_k}$, where $k$ is at most $rs$ and $X_{i_1},\dots,X_{i_k}\in\{X,X_1,\dots,X_r\}$, each one multiplied by a polynomial in $t$. This means that, if in the setting of \cref{lem:Representation} the group $\mathbb G$ is nilpotent of nilpotency step $s$, the right hand side in \eqref{eqn:RepresentationImportant} can be written as the sum of at most $rs+k-1$ mixed derivatives, in some of the directions $X,X_1,\dots,X_r$, of $f$, each one multiplied by a polynomial in $t$. Notice also that the degree of $t$ in the right hand side of \eqref{eqn:RepresentationImportant} can be at most $r(s-1)+k-1$. 
\end{remark}
\begin{lemma}\label{lem:EST1BIS}
	For each positive integers $s,k,\ell$ there exist positive integers $D, \nu$ with the following property. Let $\mathbb G$ be a nilpotent Lie group of nilpotency step $s$ and $p\in\mathbb G$. Let us assume that $f:\mathbb G\to\mathbb R$ is smooth and $k$-polynomial with respect to $Y_1,\dots,Y_\ell\in \mathfrak{g}$. Then, there exists a polynomial $P:\mathbb R^\ell \to \mathbb R$ with degree at most $\nu$, whose coefficients depend on mixed derivatives along some directions of $\{Y_1,\dots,Y_\ell\}$ of order at most $D$ of $f$ at $p$, such that
	\begin{equation}\label{eqn:YEAHBIS}
	f(p\exp(t_1Y_1)\dots\exp(t_\ell Y_\ell))= P(t_1,\dots,t_\ell),
	\end{equation}
	for every $t_1,\dots,t_\ell\in\mathbb R$. 
\end{lemma}
\begin{proof}
	If $\ell=1$ the proof is straightforward from \eqref{eqn:taylor}. Let us assume $\ell\geq 2$. Let us inductively construct the string of natural numbers $\{a_0,\dots,a_{\ell-1}\}$ as follows
	\begin{equation}\label{eqn:DefinitionOfSequenceBIS}
	\begin{aligned}
	a_{\ell-1}:=k-1, \qquad a_{j-1}:=sa_{j}+k-1, \quad \forall 1\leq j\leq \ell-1.
	\end{aligned}
	\end{equation}
	Let us inductively define the string of natural numbers $\{\nu_1,\dots,\nu_{\ell-1}\}$ as follows
	\begin{equation}\label{eqn:OtherInductiveConstantsBIS}
	\begin{aligned}
	&\nu_1:= k-1+a_1(s-1), \qquad \nu_{j+1}:=\nu_j+k-1+a_{j+1}(s-1), \quad \forall j=1,\dots,\ell-2, \\
	\end{aligned}
	\end{equation}
	where if $\ell=2$ the second part of the previous equation does not come into play. We claim that 
	\begin{equation}\label{eqn:TheConstants}
	\nu=\nu(k,s,\ell):=\nu_{\ell-1}+k-1, \qquad D=D(k,s,\ell):=a_0.
	\end{equation}
	Let us prove an intermediate result. Let us fix $\{Y_1,\dots,Y_\ell\}\subseteq S$.
	We want to prove by induction on $j=1,\dots,\ell-1$ the following statement. For every $0\leq m\leq a_j$, and for every $X_{i_1},\dots,X_{i_m}\in \{Y_1,\dots,Y_\ell\}$ there exists a polynomial $P_j:\mathbb R^j\to \mathbb R$ with degree at most $\nu_j$ whose coefficients depend on the mixed derivatives of $f$ at $p$ of order at most $a_0$ along some directions of $\{Y_1,\dots,Y_\ell\}$,  such that for every $t_1,\dots,t_j\in \mathbb R$,  the following equality holds
	\begin{equation}\label{eqn:inductionBIS}
	(X_{i_1}\dots X_{i_m}f)(p\exp(t_1Y_1)\dots\exp(t_j Y_j))=P_j(t_1,\dots,t_j).
	\end{equation}
	Let us go through the base step. Let us fix $0\leq m\leq a_1$ and $X_{i_1},\dots,X_{i_m}\in\{Y_1,\dots,Y_\ell\}$, and we want to write $(X_{i_1}\dots X_{i_m}f)(p\exp(t_1Y_1))$. In order to do this we use \eqref{eqn:RepresentationImportant} in the case $\mathbb G$ is a nilpotent group, see \cref{rem:Nilpoten}. Indeed,  from the right hand side of \eqref{eqn:RepresentationImportant} and the reasoning in \cref{rem:Nilpoten}, we get that $(X_{i_1}\dots X_{i_m}f)(p\exp(t_1Y_1))$ is a sum of at most $sm+k-1\leq sa_1+k-1=a_0$ mixed derivatives of $f$ evaluated at $p$ in some directions of the set $\{Y_1,\dots,Y_\ell\}$ each one multiplied by a polynomial in $t_1$. Moreover the degree of $t_1$ is at most $k-1+m(s-1)\leq k-1+a_1(s-1)=\nu_1$, see again the right hand side in \eqref{eqn:RepresentationImportant} and the reasoning in \cref{rem:Nilpoten}. Thus we have proved \eqref{eqn:inductionBIS} in the base step $j=1$. 
	Let us now proceed with the induction and let us assume the statement is true for some $j=1,\dots,\ell-2$. We want to prove it true for $j+1$. Thus, let us fix $0\leq m\leq a_{j+1}$ and $X_{i_1},\dots,X_{i_m}\in\{Y_1,\dots,Y_\ell\}$, and we want to write
	$$
	(X_{i_1}\dots X_{i_m}f)(\underbrace{p\exp(t_1Y_1)\dots\exp(t_j Y_j)}_{p'}\exp(t_{j+1}Y_{j+1})).
	$$
	We can thus apply \eqref{eqn:RepresentationImportant} with $p'$ instead of $q$.
	From the right hand side of \eqref{eqn:RepresentationImportant} and the reasoning in  \cref{rem:Nilpoten} we get that $(X_{i_1}\dots X_{i_m}f)(p'\exp(t_{j+1}Y_{j+1}))$ is a sum of at most $sm+k-1\leq sa_{j+1}+k-1=a_j$, see \eqref{eqn:DefinitionOfSequenceBIS}, mixed derivatives of $f$ in some directions of $\{Y_1,\dots,Y_\ell\}$ evaluated at $p'=p\exp(t_1Y_1)\dots\exp(t_jY_j)$, each one multiplied by a polynomial in $t_{j+1}$. Moreover the degree of $t_{j+1}$ is at most $k-1+m(s-1)\leq k-1+a_{j+1}(s-1)$, see again the right hand side of \eqref{eqn:RepresentationImportant} and the reasoning in \cref{rem:Nilpoten}. By the inductive hypothesis, every mixed derivative of order at most $a_j$ of $f$ in some directions of $\{Y_1,\dots,Y_\ell\}$ evaluated at $p'=p\exp(t_1Y_1)\dots\exp(t_{j} Y_{j})$ is a polynomial in $(t_1,\dots,t_{j})$ with degree at most $\nu_j$ whose coefficients depend on mixed derivatives of $f$ at $p$ of order at most $a_0$ along some directions of $\{Y_1,\dots,Y_\ell\}$. Then, for all $t_1,\dots,t_{j+1}\in\mathbb R$ we can write
	$$
	(X_{i_1}\dots X_{i_m}f)(\underbrace{p\exp(t_1Y_1)\dots\exp(t_j Y_j)}_{p'}\exp(t_{j+1}Y_{j+1}))=P_{j+1}(t_1,\dots,t_{j+1}),
	$$
	where $P_{j+1}$ is a polynomial with degree at most $\nu_j+k-1+a_{j+1}(s-1)=\nu_{j+1}$, see \eqref{eqn:OtherInductiveConstantsBIS}, and whose coefficients depend on mixed derivatives of $f$ at $p$ of order at most $a_0$ along some directions of $\{Y_1,\dots,Y_\ell\}$. Thus this concludes the proof of \eqref{eqn:inductionBIS} by induction. 
	
	Let us now complete the proof. By \eqref{eqn:taylor} we can write, for every $t_1,\dots,t_\ell\in\mathbb R$, that
	$$
	f(\underbrace{p\exp(t_1 Y_1)\dots\exp(t_{\ell-1}Y_{\ell-1})}_{p''}\exp(t_\ell Y_\ell))=\sum_{i=0}^{k-1}\frac{t_\ell^i}{i!}(Y_\ell^if)(p'').
	$$
	By the induction before, see \eqref{eqn:inductionBIS}, and since $a_{\ell-1}=k-1$, see \eqref{eqn:DefinitionOfSequenceBIS}, we conclude that, for every $0\leq i\leq k-1$, $(Y^i_\ell f)(p'')=(Y^i_\ell f)(p\exp(t_1 Y_1)\dots\exp(t_{\ell-1}Y_{\ell -1}))$ is a polynomial in $(t_1,\dots,t_{\ell-1})$ with degree at most $\nu_{\ell-1}$ whose coefficients depend on mixed derivatives of $f$ at $p$ of order at most $a_0$ along some directions of $\{Y_1,\dots,Y_\ell\}$. Thus we conclude \eqref{eqn:YEAHBIS} where $P$ is a polynomial with degree at most $\nu:=\nu_{\ell -1}+k-1$ whose coefficients depend on mixed derivatives of $f$ at $p$ of order at most $D=a_0$ along some directions of $\{Y_1,\dots,Y_\ell\}$. Thus we obtained the result with the constants chosen in \eqref{eqn:TheConstants}.
\end{proof}
\subsection{Proof of \cref{thm:Intro2}}\label{sec:4.1} Before starting the proof of \cref{thm:Intro2} we give a refinement of \cref{lem:ConnectionBIS} when $\mathbb G$ is a Carnot group.
	\begin{lemma}\label{lem:Connecting}
	Let $\mathbb G$ be a Carnot group of topological dimension $n$, let $S$ be a basis of its first layer, and let $d_{\mathbb G}$ be the subRiemannian distance associated to $S$. Then there exist $\widetilde C$, and $X_{1},\dots,X_{2n}\in S$ such that given two arbitrary points $p,q\in \mathbb G$ there exist $t_1,\dots,t_{2n}\in\mathbb R$  such that 
	$$
	q=p\exp(t_1X_{1})\dots \exp(t_{2n} X_{2n}), 
	$$
	and 
	\begin{equation}\label{eqn:Conclude}
	d_{\mathbb G}(p,q)\geq \widetilde C(\abs{t_1}+\dots+\abs{t_{2n}}).
	\end{equation}
\end{lemma}
\begin{proof}
	Without loss of generality we can prove the statement for $p=e$. From \cref{lem:ConnectionBIS}, by using the same notation therein, up to eventually reduce $\hat V$ there exists $\eps>0$ such that $(\hat\psi_e)_{|_{\hat V}}:\hat V\to B_\eps(e)$ is a diffeomorphism, where $\hat V\subseteq (0,1)^n$, and $B_\eps(e)$ is the open ball, in the metric $d_{\mathbb G}$, of radius $\eps$ and centre the identity $e$. Thus we have shown that there exist $\eps>0$ and $X_{1},\dots,X_{2n}\in S$ such that for all $q\in B_\eps(e)$ there exists at least one $(t_1,\dots,t_n)\in\hat V\subseteq (0,1)^n$ such that 
	$$
	\hat\psi_e(t_1,\dots,t_n)=q.
	$$ 
	We now claim that the Lemma holds true with $\widetilde C:=\eps/(4n)$, and with the $2n$ vector fields $X_{1},\dots,X_{2n}\in S$ found above. Let us define the norm $\|g\|_{\mathbb G}:=d_{\mathbb G}(g,e)$. Let us take an arbitrary $q\in\mathbb G$ and consider $q':=\delta_{\eps/(2\|q\|_{\mathbb G})}q$, where $\eps$ is defined above and $\delta_{\lambda}$ is the dilation of factor $\lambda$ on $\mathbb G$. Since $q'\in B_\eps(e)$ we have that there exists $(t_1,\dots,t_n)\in (0,1)^n$ such that 
	$$
	\exp(t_1X_{1})\dots\exp(t_nX_{n})\exp(-\hat s_nX_{n+1})\dots\exp(-\hat s_1X_{2n})=:\hat\psi_e(t_1,\dots,t_n)=q',
	$$
	where $\hat s\in (0,1)^n$ is as in \cref{lem:ConnectionBIS}. Thus if we dilate by a factor $k:=2\|q\|_{\mathbb G}/\eps$ the previous equality we get, since $X_1,\dots,X_n$ are in the first layer
	$$
	\exp(kt_1X_{1})\dots\exp(kt_nX_{n})\exp(-k\hat s_nX_{n+1})\dots\exp(-k\hat s_1X_{2n})=q.
	$$
	Since $(t_1,\dots,t_n)\in (0,1)^n$ and $(\hat s_1,\dots,\hat s_n)\in (0,1)^n$ we get that
	$$
	\abs{kt_1}+\dots+\abs{kt_n}+ \abs{k\hat s_1}+\dots+\abs{k\hat s_n} \leq \frac{4n}{\eps}\|q\|_{\mathbb G},
	$$
	and thus we conclude that every point $q\in\mathbb G$ can be connected to the identity $e\in\mathbb G$ by means of a concatenation of at most $2n$ horizontal lines in the directions $X_{1},\dots,X_{2n}$, and moreover \eqref{eqn:Conclude} holds with $\widetilde C=\eps/(4n)$. This concludes the proof.
\end{proof}

We now show that on an arbitrary Carnot group a $k$-polynomial distribution with respect to a basis of the first layer of the Lie algebra is polynomial in exponential chart. The forthcoming proposition is the first step to prove \cref{thm:Intro2}. Indeed the forthcoming \cref{thm:MAINCARNOT} is precisely \cref{thm:Intro2} in the setting of Carnot groups. Then we will prove \cref{thm:MAINGENERAL}, which is \cref{thm:Intro2} for simply connected nilpotent groups, and eventually we conclude this section by giving the proof of \cref{thm:Intro2}.

\begin{proposition}\label{thm:MAINCARNOT}
	Let $k,n,s$ be positive integers. Then there exists a constant $\nu$ such that the following holds. Let $\mathbb G$ be an arbitrary Carnot group of step $s$ and topological dimension $n$, and assume that $f\in\mathcal{D}'(\mathbb G)$ is a distribution that is $k$-polynomial with respect to a basis $S$ of the first layer of $\mathfrak g$. Then $f$ is a polynomial in exponential chart, see \cref{def:Polynomial}, with degree at most $\nu$.
\end{proposition}
\begin{proof}
	We stress a little abuse of notation in this proof. For a function $g:\mathbb G\to\mathbb R$ we will write without making a distinction between $g$ and $g\circ \exp$, since when $\mathbb G$ is a Carnot group $\exp$ is a global analytic diffeomorphism. In other words we will identify $\mathbb G\equiv \mathfrak g\equiv \mathbb R^n$ by means of the exponential map $\exp$ and a choice for an adapted basis of $\mathfrak g$ that extends $S$. 
	
	From \cref{thm:Intro1} we get that $f$ is represented by an analytic function. Let us fix $\Omega$ an open neighbourhood of the identity of $\mathbb G$ on which $f$ coincides with his Taylor expansion. Thus in particular we have the following equality in the pointwise sense
	\begin{equation}\label{eqn:SERIES}
	f(t_1,\dots,t_n)=\sum_{d=0}^\infty p_d(t_1,\dots,t_n), \qquad \text{for all $(t_1,\dots,t_n)\in\Omega$},
	\end{equation}
	where $p_d(t_1,\dots,t_n)$ is the polynomial of the Taylor expansion that is $\delta_\lambda$-homogeneous with degree $d$. We claim that for every $d\geq 0$ the polynomial $p_d$ is $k$-polynomial with respect to $S$. Indeed, let us prove this statement by induction on $d$. Clearly if $d=0$ the conclusion is proved since constant functions are always $k$-polynomial with respect to $S$. Let us now assume that $p_i$ is $k$-polynomial with respect to $S$ for $i=0,\dots,d$. We want to prove that $p_{d+1}$ is $k$-polynomial with respect to $S$. First of all, by linearity of the condition of being $S$-polynomial with degree at most $k$, we get that
	$$
	g_d:=f-\sum_{i=0}^d p_i \quad \text{is $k$-polynomial with respect to $S$.}
	$$
	Since $g_d$ is $k$-polynomial with respect to $S$ we also get that $g_d\circ\delta_{\lambda}$ is $k$-polynomial with respect to $S$ for every $\lambda>0$. This last assertion comes from the iteration of the equality 
	$$
	X(g_d\circ\delta_\lambda)(p)=\frac{\de}{\de\eps}_{|_{\eps=0}}(g_d\circ\delta_\lambda)(p\exp(\eps X))=\lambda (Xg_d)(\delta_\lambda(p)), \quad \text{for all $S$, $\lambda>0$, $p\in\mathbb G$}.
	$$
	From \eqref{eqn:SERIES} we get that 
	$$
	\frac{g_d\circ\delta_\lambda}{\lambda^{d+1}}=p_{d+1}+\lambda\left(\sum_{i=d+2}^{+\infty}\lambda^{i-d-2}p_i\right)=:p_{d+1}+\lambda R_{d+1}^{\lambda}, \qquad \text{on $\delta_{\lambda^{-1}}\Omega$, for all $\lambda>0$}.
	$$
	Since $g_d\circ\delta_{\lambda}$ is a $k$-polynomial function with respect to $S$  the same is true for $(g_d\circ\delta_\lambda)/\lambda^{d+1}$, and thus for $p_{d+1}+\lambda R_{d+1}^{\lambda}$, on $\delta_{\lambda^{-1}}\Omega$, for every $\lambda>0$, since the previous equality holds. Let us fix $\Omega'$ an arbitrary open bounded set of $\mathbb G$. Since $\Omega'$ is bounded there exists $\lambda_0$ such that $\Omega'\subseteq \delta_{\lambda^{-1}}\Omega$ for all $\lambda\leq \lambda_0$. Thus, for every $\lambda\leq \lambda_0$, the function $p_{d+1}+\lambda R_{d+1}^{\lambda}$ is a $k$-polynomial function with respect to $S$ on $\Omega'$. Since $p_{d+1}+\lambda R_{d+1}^{\lambda}$ converges pointwise to $p_{d+1}$ on $\Omega'$ as $\lambda\to 0$, we can apply \cref{coroll:KpolAtLimit} to infer that $p_{d+1}$ is $k$-polynomial with respect to $S$ on $\Omega'$. Since $\Omega'$ is arbitrary we get that $p_{d+1}$ is $k$-polynomial with respect to $S$ on the entire $\mathbb G$ and thus the induction is complete.
	
	Let us now prove an independent result that will lead to the conclusion of the proof. Let us prove that there exists a constant $\nu$, which depends on $k,s,n$, such that if $f$ is an arbitrary smooth function that is $k$-polynomial with respect to $S$ on $\mathbb G$, then $f(p)=O(\|p\|_{\mathbb G}^{\nu})$ as $\|p\|_{\mathbb G}\to+\infty$, where $\|\cdot\|_{\mathbb G}$ is the homogeneous norm associated to the subRiemannian distance $d_{\mathbb G}$ induced by $S$, i.e., $\|p\|_{\mathbb G}:=d_{\mathbb G}(p,e)$, where $e$ is the identity of the group. 
	
	Indeed, as a direct application of  \cref{lem:Connecting}, we first get that there exist $\widetilde C$, and some $X_{1},\dots,X_{2n}\in S$ such that for every point $p\in \mathbb G$ we can write the following equality
	$$
	p=\exp(t_1X_{1})\dots\exp(t_{2n} X_{2n}),
	$$
	for some $t_1,\dots,t_{2n}\in\mathbb R$, and moreover the following inequality holds 
	$$
	\|p\|_{\mathbb G}\geq \widetilde C(\abs{t_1}+\dots+\abs{t_{2n}}).
	$$
	Moreover, as a direct application of \cref{lem:EST1BIS}, there exist a constant $\nu$, which depends on $s,k,n$, and a constant $C$, which may depend also on $f$, such that
	$$
	\abs{f(p)}=\abs{f(\exp(t_1X_{1})\dots\exp(t_{2n} X_{2n}))}\leq C(1+\abs{t_1}+\dots+\abs{t_{2n}})^\nu\leq C(1+\|p\|_{\mathbb G}/\widetilde C)^\nu.
	$$
	Thus $f(p)=O(\|p\|_{\mathbb G}^{\nu})$, as $\|p\|_{\mathbb G}\to+\infty$. In order to conclude the proof let us prove that in the sum \eqref{eqn:SERIES}, $p_d\equiv 0$ for all $d\geq \nu+1$. Let us fix $d\geq \nu+1$ and let us suppose by contradiction that $\max_{\{\|x\|_{\mathbb G}=1\}} p_d(x)=p_d(\overline x)=m>0$, for some $\overline x\in\mathbb G$ with $\|\overline x\|_{\mathbb G}=1$. Hence for an arbitrary $\lambda>0$, since $p_d$ is a $\delta_{\lambda}$-homogeneous polynomial of degree $d$, we have $p_d(\delta_{\lambda}\overline x)=\lambda^dp_d(\overline x)=m\lambda^d$. Hence, being $d\geq \nu+1$ and $m>0$, the previous inequality is a contradiction with the fact that $p_d(x)=O(\|x\|_{\mathbb G}^{\nu})$ as $\|x\|_{\mathbb G}\to+\infty$, which holds true since we proved that $p_d$ is $k$-polynomial with respect to $S$, and since we obtained the polynomial-growth bound above. Thus we conclude that $p_d\equiv 0$ for all $d\geq \nu+1$, and then by analytic continuation $\eqref{eqn:SERIES}$ holds everywhere on $\mathbb G$, where now the sum is taken up to $\nu$. Then $f$ is a polynomial in exponential chart. Moreover, its homogeneous degree, and thus its degree, is bounded above by $\nu$.
\end{proof}
\begin{theorem}\label{thm:MAINGENERAL}
	Let $\mathbb G$ be a simply connected nilpotent group of nilpotency step $s$, and let $S\subseteq \mathfrak g$ be a subset that Lie generates $\mathfrak g$. If $f\in\mathcal{D}'(\mathbb G)$ is a distribution that is $S$-polynomial then $f$ is a polynomial in exponential chart, see \cref{def:Polynomial}.
\end{theorem}
\begin{proof}
	We stress a little abuse of notation in this proof. For a function $g:\mathbb G\to\mathbb R$ we will write without making a distinction between $g$ and $g\circ \exp$, since when $\mathbb G$ is a simply connected nilpotent group, $\exp$ is a global analytic diffeomorphism. In other words we identify $\mathbb G\equiv \mathfrak g\equiv \mathbb R^n$ by means of $\exp$ and a choice for a basis of $\mathfrak g$. Up to taking a subset of $S$ that is finite and still Lie generates $\mathfrak g$, we may assume that $S$ is finite, namely $S=\{X_1,\dots,X_m\}$ for some $X_1,\dots,X_m\in \mathfrak{g}$. Since now $S$ is finite, there exists $k\in\mathbb N$ such that $f$ is $k$-polynomial with respect to $S$.
	
	From \cref{thm:Intro1} we get that $f$ is represented by an analytic function. Let us consider the free Lie algebra $\mathfrak f_{m,s}$ of step $s$ and with $m$ generators $\{X_1',\dots,X_m'\}$ introduced in \cref{def:free}. By item (iii) of \cref{def:free} there exists a Lie algebra homomorphism $\phi:\mathfrak f_{m,s}\to\mathfrak g$ such that $\phi(X'_i)=X_i$ for every $1\leq i\leq m$.  We claim that $f\circ\phi$ is smooth and $k$-polynomial with respect to $\{X_1',\dots,X_m'\}$. This latter assertion is true since, first of all $f\circ\phi$ is smooth since $f$ is analytic and $\phi$ is linear, and second because, from the fact that $\phi(X'_i)=X_i$ for every $1\leq i\leq m$, we conclude that
	$$
	X'_i(f\circ\phi)=X_if\circ\phi, \qquad \forall 1\leq i\leq m,
	$$	
	and thus iterating 
	$$
	(X'_i)^k(f\circ\phi)=X_i^kf\circ\phi\equiv 0, \qquad \forall 1\leq i\leq m. 
	$$	
	Since $\mathrm{span}\{X_1',\dots,X_m'\}$ is the first layer of a stratification of $\mathfrak f_{m,s}$, we can apply \cref{thm:MAINCARNOT} and conclude that $f\circ\phi$ is a polynomial in exponential chart with degree at most $\delta'$, where $\delta'$ depends on $k,m,s$ since the topological dimension of $\mathfrak f_{m,s}$ is bounded above by a function of $m$ and $s$.  Since $\{X_1,\dots,X_m\}$ Lie generates $\mathfrak g$, since $\phi(X'_i)=X_i$ for every $1\leq i\leq m$, and since $\phi$ is a Lie algebra homomorphism, we get that $\phi$ is surjective. Hence there exists a linear map $\phi^{-1}:\mathfrak g\to\mathfrak f_{m,s}$ such that $\phi\circ \phi^{-1}=\mathrm{id}_{|_{\mathfrak g}}$. Thus $f=(f\circ\phi)\circ\phi^{-1}$ is a polynomial in exponential chart since it is the composition of a polynomial in exponential chart with a linear map. Notice that the degree of $f$ in exponential chart is at most $\delta'$ since $\phi$ is linear.
\end{proof}
\begin{proof}[Proof of {\cref{thm:Intro2}}]\label{proofTHMIntro1} From \cite[Theorem 3.6.1]{Varadarajan} we deduce that there exists a unique simply connected nilpotent Lie group $\mathbb G'$ with Lie algebra $\mathfrak g$, and $\mathbb G$ is the quotient of $\mathbb G'$ with one central discrete subgroup $\Gamma$ of $\mathbb G'$. Let $\pi:\mathbb G'\to \mathbb G'\mathrel{/}\Gamma \simeq \mathbb G$ be the projection map, which is open. Then $\pi_*:\mathfrak g\to\mathfrak g$ is surjective, and hence a bijection. We claim that the map $f\circ \pi$ is $(\pi_*)^{-1}(S)$-polynomial in $\mathbb G'$. Indeed, for every $X\in\mathfrak (\pi_*)^{-1}(S)$ we have 
	\begin{equation}\label{eqn:IterativelyApplying}
	X(f\circ \pi)=(\pi_*X)f\circ \pi,
	\end{equation}
	and thus iterating and using that $f$ is $S$-polynomial we get the sought claim. Hence $f\circ\pi\circ\exp_{\mathbb G'}$ is a polynomial, according to \cref{thm:MAINGENERAL}, since $(\pi_*)^{-1}(S)$ Lie generates $\mathfrak g$ as well. But since $f\circ\pi\circ\exp_{\mathbb G'}=f\circ \exp_{\mathbb G}\circ\pi_*$, and since $\pi_*$ is a bijection, we get that $f\circ\exp_{\mathbb G}$ is a polynomial as well, and then we are done. 	
\end{proof}
\begin{remark}[The constant $\nu$ in \cref{thm:MAINCARNOT}]
	We stress that from the proof of \cref{thm:MAINCARNOT} we infer that the homogeneous degree of $f$, and thus also the degree of $f$, in the exponential chart is at most $\nu(k,s,2n)$, where $\nu$ is explicitly provided in the proof of \cref{lem:EST1BIS}, see \eqref{eqn:TheConstants}. Thus the constant $\nu$ in \cref{thm:MAINCARNOT} can be taken to be $\nu(k,s,2n)$. This in particular gives, in case $\mathbb G$ is connected and nilpotent, an explicit bound on the degree of the polynomial in exponential chart that represents a distribution $f$ that is $k$-polynomial with respect to a Lie generating $S$, see the proofs of \cref{thm:MAINGENERAL} and \cref{thm:Intro2}.
\end{remark}
\begin{remark}[Relaxation of the hypotheses in \cref{thm:Intro2}]
	The hypothesis of $f\in\mathcal{D}'(\mathbb G)$ being polynomial with respect to $S$ in \cref{thm:Intro2} can be relaxed to the following one: for every $X\in S$ there exists $k\in\mathbb N$ and a polynomial in exponential chart $g$ such that $X^k f = g$ in the distributional sense on $\mathbb G$. Indeed, if this is the case, there exists a finite subset $\{X_1,\dots,X_m\}\subseteq S$ that Lie generates $\mathfrak g$ and such that $X_i^{k_i}f=g_i$ for $1\leq i\leq m$, where $k_i\in\mathbb N$ and $g_i$ are polynomials in exponential chart. Thus, taking \cref{prop:POLYKPOLY} into account, there exists a sufficiently large $k\in\mathbb N$ such that $f$ is $k$-polynomial with respect to $\{X_1,\dots,X_m\}$, and then we can use \cref{thm:Intro2}.
\end{remark}
\begin{remark}[$S$-polynomial implies $\mathfrak g$-polynomial]\label{rem:Simpliesg}
	Let us further notice the following non-obvious fact, which is a consequence of \cref{thm:Intro2} and \cref{prop:POLYKPOLY}. If $\mathbb G$ is a connected nilpotent Lie group and $S$ Lie generates $\mathfrak g$, then if $f\in\mathcal{D}'(\mathbb G)$ is $S$-polynomial, $f$ is $\mathfrak g$-polynomial, with a degree of polynomiality $k$ uniform with respect to $X\in\mathfrak g$, but that may eventually depend on $f$. This latter assertion is true since if $f$ is $S$-polynomial, then \cref{thm:Intro2} tells us that $f$ is a polynomial in exponential chart and thus we can apply \cref{prop:POLYKPOLY}.
\end{remark}
\begin{remark}
The example in \eqref{ex:PositiveAffine} shows that our \cref{thm:Intro2} is sharp in the class of connected nilpotent Lie groups.
\end{remark}
\begin{remark}
	In case $\mathbb G$ is simply connected and nilpotent, $\exp$ is a global analytic diffeomorphism and so the class of polynomial distributions à la Leibman on $\mathbb G$ coincides, up to identifying $\mathbb G\equiv \mathfrak g$, with the class of polynomials on $\mathbb R^n$, where $n$ is the topological dimension of $\mathbb G$. If $\mathbb G$ is nilpotent but not necessarily simply connected it might happen that the class of polynomial trivializes: for example it is readily seen that the polynomials on a torus $\mathbb S^1\times \mathbb S^1$ are just the constant functions.
\end{remark}

\section{Relations between various notions of polynomial}\label{sec:App}
In this section we shall prove \cref{prop:Equ1Intro} and \cref{prop:Equ2Intro}.

\subsection{Proof of \cref{prop:Equ1Intro}}
Let us now prove \cref{prop:Equ1Intro}, that is, let us prove that our definition of polynomial distribution à la Leibman, see \cref{rem:LeiBravo}, is consistent with our definition of polynomial distribution, see \cref{def:PolyGeneral}, on every connected Lie group. We first prove this equivalence on smooth functions, and then conclude by using convolutions with smooth kernels.

\begin{proposition}\label{rem:PolyLeib}
	Let $\mathbb G$ be a connected Lie group and let $\phi:\mathbb G\to\mathbb R$ be a smooth function. Then $\phi$ is a polynomial with degree at most $d$ à la Leibman, see \eqref{eqn:gi}, if and only if for every $X_1,\dots,X_{d+1}\in\mathfrak g$ we have 
	\begin{equation}\label{eqn:POLYDISTR}
	X_1\dots X_{d+1}\phi\equiv 0, \qquad \text{on $\mathbb G$}.
	\end{equation}
\end{proposition}
\begin{proof}
	First, let us prove that \eqref{eqn:gi} implies that whenever $X_1,\dots,X_{d+1}\in\mathfrak g$ then \\ $X_1\dots X_{d+1}\phi \equiv 0$ on $\mathbb G$. In order to show this, let us notice that, defined $g_{d+1}(t):=\exp(tX_{d+1})$, then $t^{-1}D_{g_{d+1}(t)}\phi$ converges pointwise on $\mathbb G$ to $X_{d+1}\phi$ as $t\to 0$. Moreover, for every $g\in\mathbb G$, the operator $D_g$ is continuous with respect to the pointwise convergence of functions, i.e., if $\phi_n\to\phi$ pointwise on $\mathbb G$ as $n\to+\infty$, then $D_g\phi_n\to D_g\phi$ pointwise on $\mathbb G$ as $n\to+\infty$, for every $g\in\mathbb G$. By also using that $D_g(\lambda\phi)=\lambda D_g(\phi)$ for every $g\in\mathbb G$ and $\lambda\in\mathbb R$, and putting $g_{d+1}(t)$ in \eqref{eqn:gi} we get 
	$$
	D_{g_1}\dots D_{g_d}(t^{-1}D_{g_{d+1}(t)}\phi)\equiv 0, \quad \text{for all $t>0$ on $\mathbb G$} \Rightarrow D_{g_1}\dots D_{g_{d}}X_{d+1}\phi\equiv 0,
	$$  
	for every $g_1,\dots,g_{d}\in \mathbb G$, where in the previous conclusion we are taking $t\to 0$ and we are exploiting the continuity of the operators $D_g$ with respect to the pointwise convergence. Now if we iterate the argument with $X_{d+1}\phi$ instead of $\phi$, we obtain
	$$
	X_1\dots X_{d+1}\phi\equiv 0, \qquad \text{on $\mathbb G$},
	$$
	which is what we wanted. 
	
	Regarding the opposite direction, let us now denote $\mathscr P^d_{\mathcal{L}}$ the vector space of smooth functions on $\mathbb G$ that are polynomials à la Leibman with degree at most $d$, see \eqref{eqn:gi}. Let us denote $\mathscr{P}^d_{\mathcal{D}}$ the vector space of smooth functions $\phi$ such that for every $X_1,\dots,X_{d+1}\in\mathfrak g$ we have $X_1\dots X_{d+1}\phi\equiv 0$ on $\mathbb G$.
	Let us now prove the following statement
	\begin{equation}\label{eqn:InductionOnD}
	\begin{split}
	&\text{for every $d\in\{0,1,\dots\}$, $\mathscr P^d_{\mathcal{L}}=\mathscr P^d_{\mathcal{D}}$, and for every $\phi\in \mathscr P^d_{\mathcal{D}}$ and every $Y\in\mathfrak g$} \\  
	&\text{the map $x \xmapsto{\phi\circ R_{\exp(Y)}} \phi(x\exp(Y))$ defined on $\mathbb G$ is in $\mathscr P^d_{\mathcal{D}}$,}
	\end{split}
	\end{equation}
	where we recall that $R_g$ stands for the right translation by $g\in\mathbb G$. It is readily seen that the first part of the previous statement proves the proposition. Let us prove \eqref{eqn:InductionOnD} by induction on $d$.
	
	If $d=0$, one readily sees that $\mathscr P^0_{\mathcal{L}}=\mathscr{P}^0_{\mathcal{D}}$ and they agree with the set of constant functions on $\mathbb G$. Thus the statement \eqref{eqn:InductionOnD} is verified for $d=0$. Let us now assume that \eqref{eqn:InductionOnD} is true for $d-1$, with $d\geq 1$, and let us prove it true for $d$. Let us start from the second part of the statement \eqref{eqn:InductionOnD}. Let us fix $Y\in\mathfrak g$ and $\phi\in\mathscr P^d_{\mathcal{D}}$. We have, for $x\in\mathbb G$, and $X\in\mathfrak g$,
	\begin{equation}\label{eqn:Rightpoly}
	\begin{split}
	X(\phi\circ R_{\exp(Y)})(x)&=\frac{\de}{\de\varepsilon}_{|_{\varepsilon=0}} \phi\circ R_{\exp(Y)}(x\exp(\varepsilon X)) = \\
	&= \frac{\de}{\de\varepsilon}_{|_{\varepsilon=0}} \phi(x\exp(Y)\exp(-Y)\exp(\varepsilon X)\exp(Y))= \\
	&=\left(\mathrm{Ad}_{\exp(-Y)}(X)\phi\right)(x\exp(Y))=\mathrm{Ad}_{\exp(-Y)}(X)\phi\circ R_{\exp(Y)}(x).
	\end{split}
	\end{equation}
	Since $\phi\in\mathscr P^d_{\mathcal{D}}$ we get that $\mathrm{Ad}_{\exp(-Y)}(X)\phi$ is in $\mathscr{P}^{d-1}_{\mathcal{D}}$ by the definition of $\mathscr{P}^d_{\mathcal{D}}$. Thus by the inductive hypothesis in the second part of the statement \eqref{eqn:InductionOnD} we get that $\mathrm{Ad}_{\exp(-Y)}(X)\phi\circ R_{\exp(Y)}\in\mathscr{P}^{d-1}_{\mathcal{D}}$. Thus by \eqref{eqn:Rightpoly} we get that $X(\phi\circ R_{\exp(Y)})\in\mathscr P^{d-1}_{\mathcal{D}}$ and then, by arbitrariness of $X$, we conclude $\phi\circ R_{\exp(Y)}\in\mathscr P^{d}_{\mathcal{D}}$. This conclude the induction for the second part of \eqref{eqn:InductionOnD}. Let us now complete the induction by proving the first part of \eqref{eqn:InductionOnD} with $d\geq 1$, assuming it is true for $d-1$.
	
	First of all, the first argument in the proof of this proposition shows that $\mathscr{P}_{\mathcal{L}}^d\subseteq \mathscr{P}_{\mathcal{D}}^d$. Let us prove $\mathscr{P}_{\mathcal{L}}^d\supseteq \mathscr{P}_{\mathcal{D}}^d$. Take $\phi\in \mathscr{P}_{\mathcal{D}}^d$. From \cref{thm:Intro1} we conclude that $\phi$ is analytic. In order to prove that $\phi\in\mathscr{P}^d_{\mathcal{L}}$ we claim that it suffices to prove that 
	\begin{equation}\label{eqn:StrengthenedAnalytic}
	D_{g_1}\cdots D_{g_d}D_{\exp(Y)}\phi\equiv 0, \qquad \text{on $\mathbb G$},
	\end{equation}
	for every $g_1,\dots,g_d\in\mathbb G$ and every $Y\in \mathfrak g$. Indeed, the map $(g_1,\dots,g_d,g_{d+1},g)\mapsto (D_{g_1}\dots D_{g_{d}}D_{g_{d+1}}\phi)(g)$ is analytic from $\mathbb G^{d+2}$ to $\mathbb R$, since $\phi$ and the group operation $\cdot$ are analytic, recall \eqref{eqn:DefinitionDg}. Thus, since the set $\{(g_1,\dots,g_d,g,\exp(Y)):g_1,\dots,g_d,g\in\mathbb G, Y\in\mathfrak g\}$ contains an open neighbourhood of the identity in $\mathbb G^{d+2}$, if we show \eqref{eqn:StrengthenedAnalytic} we are done by analytic continuation. Let us prove \eqref{eqn:StrengthenedAnalytic}.
	
	Since $\phi \in \mathscr P^d_{\mathcal{D}}$, we get that $Y\phi \in \mathscr P^{d-1}_{\mathcal{D}}$. We also have 
	\begin{equation}\label{eqn:LAFINALE}
	D_{\exp(Y)}\phi(x)=\phi(x\exp(Y))-\phi(x)=\int_0^1 (Y\phi)(x\exp(tY))\de t, \qquad \text{for all $x\in\mathbb G$},
	\end{equation}
	and $(Y\phi)\circ R_{\exp(tY)}\in \mathscr{P}^{d-1}_{\mathcal{D}}$ for every $t\in[0,1]$, since $Y\phi\in\mathscr{P}^{d-1}_{\mathcal{D}}$ and since the second part of \eqref{eqn:InductionOnD}, which we already proved, holds. Thus, by the inductive hypothesis, $(Y\phi)\circ R_{\exp(tY)}\in \mathscr{P}^{d-1}_{\mathcal{L}}$ for every $t\in[0,1]$, and since $\mathscr{P}^{d-1}_{\mathcal{L}}$ is a vector space closed under pointwise convergence, we finally get that $D_{\exp(Y)}\phi \in\mathscr{P}^{d-1}_{\mathcal{L}}$ by exploiting \eqref{eqn:LAFINALE}. By the definition of $\mathscr{P}^{d-1}_{\mathcal{L}}$ the latter conclusion implies \eqref{eqn:StrengthenedAnalytic}, and thus the induction, and the proof, are concluded. 
\end{proof} 	
\begin{proof}[Proof of {\cref{prop:Equ1Intro}}]\label{proof:EquIntro1}
	Let $f$ be a distribution such that there exists $d\in\mathbb N$ for which $X_1\dots X_{d+1}f\equiv 0$ in the sense of distributions on $\mathbb G$ whenever $X_1,\dots,X_{d+1}\in\mathfrak g$. In particular, the distribution $f$ is $\mathfrak g$-polynomial with degree at most $d+1$ and thus, from \cref{thm:Intro1}, we deduce that it is represented by an analytic function. Since in particular $f$ is represented by a smooth function we can thus use \cref{rem:PolyLeib} to obtain that $f$ is a polynomial distribution à la Leibman with degree at most $d$.
	
	Viceversa, let $f$ be a distribution that is polynomial à la Leibman with degree at most $d$. Let $\{\phi_{n}\}_{n\in\mathbb N}$ be an approximate identity. Then, iteratively applying \eqref{eqn:DgConvolution}, the convolution $\phi_n\ast f$ is a smooth function that is polynomial à la Leibman with degree at most $d$. Thus, from \cref{rem:PolyLeib}, we get that for every $n\in\mathbb N$ and every $X_1,\dots,X_{d+1}\in\mathfrak g$ we have 
	$$
	0\equiv X_1\dots X_{d+1}(\phi_n\ast f)=\phi_{n}\ast X_1\dots X_{d+1}f, \qquad \text{on $\mathbb G$},
	$$
	where in the second equality we are iteratively applying \eqref{rem:DerivativeOfAConvolution}. Thus, letting $n\to+\infty$ in the previous equality, since $\{\phi_{n}\}_{n\in\mathbb N}$ is an approximate identity we get that $X_1\dots X_{d+1}f\equiv 0$ in the sense of distributions on $\mathbb G$ whenever $X_1,\dots,X_{d+1}\in\mathfrak g$, concluding the proof of the equivalence. 
	
	The fact that $f$ is represented by an analytic function and the conclusion about the finite dimension of the vector space of the polynomial distributions with degree at most $d\in\mathbb N$ à la Leibman is now a direct consequence of \cref{thm:Intro1}. 
	
	Let us prove the final part of the statement of \cref{prop:Equ1Intro}. Let us recall that the lower central series is defined inductively as follows: $\mathfrak g_0:=\mathfrak g$, and $\mathfrak g_{k}:=[\mathfrak g,\mathfrak g_{k-1}]$ for every $k\geq 1$.
	We now prove that the distributions $f$ that are polynomial à la Leibman with degree at most $d\in\mathbb N$ are invariant along the directions of $\mathfrak g_d$, namely for every $X\in\mathfrak g_d$ we have $Xf=0$ on $\mathbb G$. 
	
	Indeed, for every $d\geq 0$, every element of $\mathfrak g_d$ can be written as a linear combination of left-invariant operators that are the composition of at least $d+1$ left-invariant vector fields. To see the latter property we proceed by induction on $d$. The case $d=0$ is true by definition, so let us suppose the assertion true for $d\geq 0$ and prove it for $d+1$. If $X\in\mathfrak g_{d+1}=[\mathfrak g,\mathfrak g_d]$, then for some $n\in\mathbb N$, $X=\sum_{i=1}^nc_{i}[X_{i},Y_{i}]$, where $X_i\in\mathfrak g$ and $Y_i\in\mathfrak g_{d}$. By the inductive step, for every $1\leq i\leq n$ there exists $j_i\in\mathbb N$ such that for every $1\leq k\leq j_i$ there exists integers $m_{i,k}\geq d+1$ and real numbers $\beta_{i,k}$ such that $Y_i=\sum_{k=1}^{j_i}\beta_{i,k} X_{i,k,1}\dots X_{i,k,m_{i,k}}$, with $X_{i,k,1},\dots,X_{i,k,m_{i,k}}\in\mathfrak g$. Thus expanding the commutator in the equality $X=\sum_{i=1}^nc_{i}[X_{i},Y_{i}]$ and by using the previous equalities on $Y_i$ we conclude. As a consequence, since every polynomial distribution à La Leibman with degree at most $d\in\mathbb N$ is such that $X_1\dots X_{d+1}f=0$ on $\mathbb G$ for every $X_1,\dots,X_{d+1}\in\mathfrak g$, see \cref{prop:Equ1Intro}, and since every $X\in\mathfrak g_d$ can be written as a linear combination of left-invariant operators that are composition of at least $d+1$ elements of $\mathfrak g$, we conclude that $Xf=0$ for every $X\in\mathfrak g_d$. 
	
	Thus, since the nilpotent residual $\mathfrak g_{\infty}$ equals the intersection $\cap_{k\in\mathbb N} \mathfrak g_k$, the previous reasoning shows that every distribution $f$ that is polynomial à la Leibman of an arbitrary degree on $\mathbb G$ is $\mathfrak g_{\infty}$-invariant, namely $Xf=0$ for every $X\in\mathfrak g_{\infty}$. As a consequence we conclude that every polynomial à la Leibman on a connected Lie group passes to the quotient to a polynomial à la Leibman on the nilpotent group $\mathbb G\mathrel{/}\mathbb G_{\infty}$, where $\mathbb G_{\infty}$ is the closure of the unique connected (and normal, since $\mathfrak g_{\infty}$ is an ideal) Lie subgroup of $\mathbb G$ with Lie algebra $\mathfrak g_{\infty}$.
\end{proof}
\begin{remark}[A variant of the first part of \cref{prop:Equ1Intro}]
	We notice that the proofs provided for the first part of \cref{prop:Equ1Intro} and for \cref{rem:PolyLeib} can be exploited, with very little modifications,  to prove the following variant of the first part of \cref{prop:Equ1Intro}. Let $\phi$ be a distribution on $\mathbb G$, a connected Lie group with Lie algebra $\mathfrak g$, and let $S\subseteq \mathfrak g$ be an $\mathrm{Ad}$-closed cone, i.e., for every $X\in S$ then $tX\in S$ for every $t\in\mathbb R$, and for every $X,Y\in S$ we have $\mathrm{Ad}_{\exp(X)}Y\in S$. Then \eqref{eqn:POLYDISTR} holds for every $X_1,\dots,X_{d+1}\in S$ if and only if \eqref{eqn:gi} holds for every $g_1,\dots,g_{d+1}\in\exp S$.
\end{remark}

\subsection{Proof of \cref{prop:Equ2Intro}}
The following proposition shows that when $\mathbb G$ is a connected nilpotent Lie group, a polynomial in exponential chart, see \cref{def:Polynomial}, is $k$-polynomial with respect to $\mathfrak g$ for some $k\in\mathbb N$. This is the last main step in order to obtain \cref{prop:Equ2Intro}.
\begin{proposition}\label{prop:POLYKPOLY}
	Let $\mathbb G$ be a connected nilpotent Lie group and let $f\colon\mathbb G\to\mathbb R$ be a polynomial in exponential chart on $\mathbb G$, see \cref{def:Polynomial}. Then there exists $k_0\in\mathbb N$ such that for every $Y_1,\dots,Y_{k_0}\in\mathfrak g$ we have $Y_1\dots Y_{k_0}f\equiv 0$ on $\mathbb G$. 
\end{proposition}
\begin{proof}
	Let us first prove the statement when $\mathbb G$ is simply connected. In the latter case $\exp$ is a global analytic diffeomorphism, hence we will abuse a little the notation in the proof of this case: for a function $g:\mathbb G\to\mathbb R$ we will write without making a distinction between $g$ and $g\circ \exp$. 
	Let $\{X_1,\dots,X_n\}$ be a basis of $\mathfrak g$, and let $s$ be the step of nilpotency of $\mathbb G$. From the definition of free-nilpotent Lie algebras we get that there exists a Lie algebra homomorphism $\phi:\mathfrak f _{n,s}\to\mathfrak g$ such that $\phi(X_i')=X_i$ for every $1\leq i\leq n$, and we recall that $V_1':=\mathrm{span}\{X_1',\dots,X_n'\}$ is the first layer of a stratification of $\mathfrak f_{n,s}$, where $X_1',\dots,X'_n$ are the generators of $\mathfrak f_{n,s}$. We stress a little abuse of notation: we will denote with $\phi$ also the map $\exp\circ \phi\circ\exp^{-1}:\mathbb F_{n,s}\to\mathbb G$.
	
	We claim that there exists $k_0\in\mathbb N$ such that for every $Y_1',\dots,Y_{k_0}'\in\{X_1',\dots,X_n'\}$, we have $Y_1'\dots Y_{k_0}'(f\circ\phi)\equiv 0$ on $\mathbb F_{n,s}$.  Indeed, since $\mathfrak f_{n,s}$ is a stratified Lie algebra with $V_1'$ as a first layer, we get that, in exponential coordinates, every $X_i'$, with $1\leq i\leq n$, is an operator of homogeneous degree $-1$; that is to say if $p$ is a polynomial in exponential chart on $\mathbb F_{n,s}$ of homogeneous degree $d$, see the last part of \cref{sec:NilpAndCarnot} for the definition of homogeneous degree, then 
	\begin{equation}\label{eqn:ACC}
	\text{the homogeneous degree of $X_i'\,p$ is less or equal than $d-1$ for every $1\leq i\leq n$}.
	\end{equation}
	The latter assertion is a simple consequence of the explicit expression of $X'_i$, for every $1\leq i\leq n$, in exponential coordinates, see \cite[Proposition 1.26]{FS82}. 
	In conclusion, since $f$ is a polynomial in exponential chart and $\phi$ is a linear map, we get that $f\circ\phi$ is a polynomial in exponential chart as well. Thus the claim is true taking into account \eqref{eqn:ACC}, and setting $k_0$ to be strictly greater than the maximum of the homogeneous degrees of the monomials of $f\circ\phi$.
	
	Now we claim that for every $Y_1,\dots,Y_{k_0}\in\{X_1,\dots,X_n\}$, we have $Y_1\dots Y_{k_0}f\equiv 0$ on $\mathbb G$. Indeed, since $\phi(X_i')=X_i$ for every $1\leq i\leq n$, we conclude that $X_if\circ\phi=X_i'(f\circ\phi)$ on $\mathbb F_{n,s}$ for every $1\leq i\leq n$. Thus, iterating, we obtain that for every $Y_1,\dots,Y_{k_0}\in\{X_1,\dots,X_n\}$ we have 
	$$
	Y_1\dots Y_{k_0}f \circ \phi = Y_1'\dots Y_{k_0}'(f\circ\phi), \qquad \text{on $\mathbb F_{n,s}$}.
	$$
	Thus, since $\phi$ is surjective, the previous equality and the first claim proven above imply the latter claim. The proof of the proposition in the case $\mathbb G$ is simply connected thus follows from the latter claim taking into account that every $Y\in\mathfrak g$ can be written as a linear combination of $X_1,\dots,X_n$. 
	
	Let us now deal with the general case in which $\mathbb G$ is connected. As at the beginning of the proof of \cref{thm:Intro2} we have a unique simply connected nilpotent $\mathbb G'$, with Lie algebra $\mathfrak g$, such that $\mathbb G$ is the quotient of $\mathbb G'$ with one central discrete subgroup $\Gamma$ of $\mathbb G'$. Let $\pi:\mathbb G'\to \mathbb G'\mathrel{/}\Gamma \simeq \mathbb G$ be the projection map, and then one has that $\pi_*:\mathfrak g\to\mathfrak g$ is a bijection. If $f:\mathbb G\to\mathbb R$ is such that $f\circ\exp_{\mathbb G}$ is a polynomial, then also $f\circ\exp_{\mathbb G}\circ\pi_*=f\circ\pi\circ\exp_{\mathbb G'}$ is a polynomial, since $\pi_*$ is bijective. Hence $f\circ\pi$ is polynomial in exponential chart in $\mathbb G'$ and we can apply the first part of this proof to obtain that there exists $k_0$ such that for every $Y_1',\dots,Y_{k_0}'\in \mathrm{Lie}(\mathbb G')\simeq \mathfrak g$ we have $Y_1'\dots Y_{k_0}'(f\circ\pi)\equiv 0$ on $\mathbb G'$. Thus, iteratively applying \eqref{eqn:IterativelyApplying}, and by using that $\pi_*$ is a bijection and $\pi$ is surjective we conclude that for every $Y_1,\dots,Y_{k_0}\in\mathrm{Lie}(\mathbb G)\simeq \mathfrak g$ we have $Y_1\dots Y_{k_0}f\equiv 0$ on $\mathbb G$, that is the sought conclusion.
\end{proof}
We now provide the proof of \cref{prop:Equ2Intro}, and we conclude with a remark.
\begin{proof}[Proof of {\cref{prop:Equ2Intro}}]\label{proof:Equ2Intro}
	(1)$\Rightarrow$(3) is a direct consequence of \cref{thm:Intro2}. (3)$\Rightarrow$(4) is a direct consequence of \cref{prop:POLYKPOLY}. (4)$\Rightarrow$(2) and (2)$\Rightarrow$(1) are trivial by definitions. (4)$\Leftrightarrow$(5) is \cref{prop:Equ1Intro}.
\end{proof}

\begin{remark}[Comparison with the results in \cite{KP20} and \cite{BLU07}]\label{rem:BLUeKP}
	We stress that a slightly weaker statement of the equivalence of (3)$\Leftrightarrow$(5) of \cref{prop:Equ2Intro} has recently appeared in \cite{KP20}. Indeed, in \cite[Theorem C]{KP20} the authors prove the equivalence between being a continuous polynomial map à la Leibman and being a polynomial in exponential chart, in the setting of simply connected nilpotent Lie groups. The proof given there is algebraic and is completely different from ours. Let us moreover notice that we do not ask for the continuity of $f$ in (5) of \cref{prop:Equ2Intro}, but we work with distributions. 
	
	Let us finally stress that, without some regularity assumption on $f$, it is not true that every polynomial map $f:\mathbb G\to\mathbb R$ à la Leibman, see \cref{rem:LeiBravo}, is smooth. Indeed, there exist non-continuous homomorphisms, and thus polynomial maps with degree at most $2$ à la Leibman, from $(\mathbb R,+)$ to $(\mathbb R,+)$.
	
	Let us also stress that the equivalent notions of being polynomial on connected nilpotent Lie groups in \cref{prop:Equ2Intro} agree with the one given in the Carnot setting in \cite[Definition 20.1.1]{BLU07}. We also notice that \cref{thm:Intro2} is a sharpening of a result contained in \cite{BLU07}. Let $\mathbb G$ be an arbitrary Carnot group of step $s$ with stratification $\mathfrak g=V_1\oplus\dots\oplus V_s$ and let $S:=\{X_1,\dots,X_m\}$ be a basis of $V_1$. In \cite[Corollary 20.1.10]{BLU07} it is proved that if there exists a smooth function $f:\mathbb G\to\mathbb R$ and a natural number $d$ such that for every $X_1,\dots,X_d\in S$ we have $X_1\dots X_d f\equiv 0$, then $f$ is a polynomial in exponential chart. Our result \cref{thm:Intro2} improves this criterion in the nilpotent case by only asking that a priori $f$ could be a distribution, and without asking anything on the mixed derivatives; i.e., it suffices that for every $1\leq i\leq m$ there exists $d_i$ such that $X_i^{d_i}f\equiv 0$ in the sense of distributions on $\mathbb G$.
\end{remark}
\appendix
\section{Examples}\label{sec:examples}
We list here some explicit examples of $S$-polynomial functions in some Lie groups. If an adapted basis $(X_1,\dots,X_n)$ of the Lie algebra $\mathfrak g$ of a Carnot group $\mathbb G$ is fixed, when we say that {\em we work in exponential coordinates of the second kind} we mean that we are identifying a point $x\in\mathbb G$ with a point of $\mathbb R^n$ as follows
$$
x\equiv (x_1,\dots,x_n)\leftrightarrow \exp(x_nX_n)\cdot\dots\cdot\exp(x_1X_1).
$$
On the contrary, when we say that {\em we work in exponential coordinates of the first kind} we mean that we are identifying a point $a\in\mathbb G$ with a point in $\mathbb R^n$ as follows
$$
a\equiv (a_1,\dots,a_n)\leftrightarrow \exp(a_1X_1+\dots+a_nX_n).
$$
\vspace{0.2cm}
\textbf{Heisenberg group.} Let $\mathbb H^1$ be the first Heisenberg group with Lie algebra 
$$
\mathfrak h^1=\mathrm{span}\{X_1,X_2\}\oplus\mathrm{span}\{X_3\}=V_1\oplus V_2,,
$$
with the only nontrivial bracket relation $[X_1,X_2]=X_3$. If we work in exponential coordinates of the second kind $(x_1,x_2,x_3)$ with respect to the adapted basis $(X_1,X_2,X_3)$ we can write
$$
X_1=\partial_1, \qquad X_2=\partial_2+x_1\partial_3, \qquad X_3=\partial_3,
$$
see \cite[page 11]{LDT20}.
Every distribution $f$ such that $X_1^2f=X_2^2f=0$ on $\mathbb G$ is represented by a polynomial, see \cref{thm:Intro2}, and moreover one can check with straightforward computations by using the expressions of the vector fields above that  $f\in\mathrm{span}\{1,x_1,x_2,x_3,x_1x_2,x_1x_3\}$. Notice that in this case, for every such $f$, $X_3^2f=0$. Nevertheless it is not true that every $\{X,Y\}$-affine function is affine along every direction of the algebra: indeed, $(X_1+X_2)^k(x_1x_3)\not\equiv 0$ for all $k\leq 3$.

If in addition to $X_1^2f=X_2^2f=0$ we ask that $(X_1X_2+X_2X_1)f=0$, the two conditions together being equivalent to asking that $X^2f=0$ for every $X\in V_1$, we conclude that $f\in\mathrm{span}\{1,x_1,x_2,x_3-(1/2)x_1x_2\}$. Notice that, when read in exponential coordinates of the first kind, the functions $x_1,x_2,x_3-(1/2)x_1x_2$ are precisely the coordinate functions $a_1,a_2,a_3$, respectively. In this way we recover the already known property that every horizontally affine function in $\mathbb H^1$ is actually affine in exponential coordinates of the first kind. For the complete characterization of horizontally affine maps in Carnot groups of step $2$ one can see \cite{LDMR20}.

\vspace{0.2cm}
\textbf{Engel group.} Let $\mathbb E^1$  be the Engel group, i.e., the Carnot group of topological dimension $4$ with stratified algebra
$$
\mathfrak e^1=\mathrm{span}\{X_1,X_2\}\oplus\mathrm{span}\{X_3\}\oplus\mathrm{span}\{X_{4}\},
$$
the only nontrivial bracket relations being $[X_1,X_2]=X_{3}$, and $[X_1,X_3]=X_4$. Working in exponential coordinates of the second kind with respect to the adapted basis $(X_1,X_2,X_3,X_4)$ we can write
$$
X_1=\partial_1, \qquad X_2=\partial_2+x_1\partial_3+(x_1^2/2)\partial_4, \qquad X_3=\partial_3+x_1\partial_4, \qquad X_4=\partial_4,
$$
see \cite[page 13]{LDT20}. We notice that a distribution $f$ is a horizontally affine function on $\mathbb E^1$, i.e., such that $X^2f=0$ for every $X\in\mathrm{span}\{X_1,X_2\}$, if and only if it satisfies the three equalities $X_1^2f=(X_1X_2+X_2X_1)f=X_2^2f=0$ on $\mathbb E^1$. Let us write explicitly the horizontally affine maps in exponential coordinates of the first kind. 

First notice that $X_3=[X_1,X_2]=X_1X_2-X_2X_1$ and $X_4=[X_1,X_3]=X_1^2X_2-2X_1X_2X_1+X_2X_1^2$. Hence, since $f$ is horizontally affine, $X_3f=2X_1X_2f$, by exploiting that $(X_1X_2+X_2X_1)f=0$. Moreover, exploiting $(X_1X_2+X_2X_1)f=X_1^2f=0$ we get $X_4f=3X_1^2X_2f$. 

From $[X_1,X_4]=0$ we deduce $0=[X_1,X_4]f=(3X_1^3X_2-X_1^2X_2X_1)f=4X_1^3X_2f$, where in the second equality we are using that $X_1^2f=0$ and in the third one we are using that $(X_1X_2+X_2X_1)f=0$. Thus $X_1^3X_2f=0$. From $[X_2,X_3]=0$ we deduce $0=[X_2,X_3]f=(2X_2X_1X_2+X_2X_1X_2)f=3X_2X_1X_2f$, where in the second equality we are using that $X_2^2f=0$. Hence $X_2X_1X_2f=0$. From $[X_2,X_4]=0$ we deduce $0=[X_2,X_4]f=(3X_2X_1^2X_2+2X_1X_2X_1X_2-X_2X_1^2X_2)f=2X_2X_1^2X_2f$, where in the second equality we are using $X_2^2f=0$ and in the third one we are using $X_2X_1X_2f=0$, which we obtained before. Then $X_2X_1^2X_2f=0$. 

Since $X_1(X_1^2X_2f)=X_2(X_1^2X_2f)=0$, we get that $X_1^2X_2f$ is constant. Thus there exists $k\in\mathbb R$ such that $X_1^2X_2f\equiv k$. Then $X_1(X_1X_2f)\equiv k$ and $X_2(X_1X_2f)=0$ readily imply that that there exists $h\in\mathbb R$ such that $X_1X_2f=kx_1+h$ in exponential coordinates of the second kind described above. Thus $X_1(X_2f)=kx_1+h$ and $X_2(X_2f)=0$ readily imply that there exists $v\in\mathbb R$ such that $X_2f=kx_1^2/2+hx_1+v$. Moreover, $X_2(X_1f)=-X_1X_2f=-kx_1-h$ and $X_1(X_1f)=0$ yield $X_3(X_1f)\equiv -k$ and $X_4(X_1f)\equiv 0$ so that by integrating the system of PDEs we obtain that there exists $m\in\mathbb R$ such that $X_1f=-kx_3-hx_2+m$.

Thus if $f$ is horizontally affine, there exist $k,h,v,m\in\mathbb R$ such that $X_2f=kx_1^2/2+hx_1+v$ and $X_1f=-kx_3-hx_2+m$. Thus we obtain that $X_3f=(X_1X_2-X_2X_1)f=2kx_1+2h$ and $X_4f=(X_1X_3-X_3X_1)f=3k$. Thus integrating the system of PDEs one obtains that there exists $n\in\mathbb R$ such that $f=k(3x_4-x_1x_3)+h(2x_3-x_1x_2)+vx_2+mx_1+n$. Thus if $f$ is horizontally affine, $f\in\mathrm{span}\{1,x_1,x_2,2x_3-x_1x_2,3x_4-x_1x_3\}$ and it is readily verified that each element of the previous vector space is actually horizontally affine, and thus this is a characterization of the horizontally affine maps in $\mathbb E^1$.

We can check, through simple computations involving BCH formula, that 
$$
\exp(x_4X_4)\exp(x_3X_3)\exp(x_2X_2)\exp(x_1X_1)=\exp(a_1X_1+a_2X_2+a_3X_3+a_4X_4),
$$
implies that $x_1=a_1$, $x_2=a_2$, $x_3=a_3+a_1a_2/2$, and $x_4=a_4+a_1a_3/2+a_1^2a_2/6$. Thus one obtains, by using exponential coordinates of the first kind, that $f$ is horizontally affine if and only if $f\in\mathrm{span}\{1,a_1,a_2,a_3,6a_4+a_1a_3\}$. As a consequence, already in the easiest step-3 Carnot group, one has a horizontally affine function that is not affine in exponential coordinates of the first kind, namely $\widetilde f(a_1,a_2,a_3,a_4):=6a_4+a_1a_3$. As a consequence the sublevel sets of $\widetilde f$ are precisely monotone sets that are not half-spaces.

\vspace{0.3cm}
One can also write down the explicit expression of a family of $\{X_1,X_2\}$-polynomial distributions $f$. It can be proved through some computations involving the explicit expressions of $X_1,X_2$ above that the vector space of the distributions $f$ on $\mathbb E^1$ such that $X_1f=X_2^2f=0$ is ${\rm span}\{1,x_2,x_3,x_4,x_2x_4-x_3^2/2\}$, where the functions are written in exponential coordinates of the second kind associated to $(X_1,X_2,X_3,X_4)$.

\vspace{0.2cm}
\textbf{Free group of step 3 and rank 2.} Let $\mathbb F_{23}$ be the free Carnot group of step 3 and rank 2, with Lie algebra $\mathfrak f_{23}$ equipped with the stratification
$$
\mathfrak f_{23}:=\mathrm{ span}\{X_1,X_2\}\oplus\mathrm {span}\{X_3\}\oplus\mathrm {span}\{X_4,X_5\},
$$
with nontrivial bracket relations being $[X_2,X_1]=X_3$, $[X_3,X_1]=X_4$, $[X_3,X_2]=X_5$. In exponential coordinates of the second kind associated to the adapted basis $(X_1,X_2,X_3,X_4,X_5)$ we can write
$$
X_1=\partial_1, \quad X_2=\partial_2-x_1\partial_3+(x_1^2/2)\partial_4+x_1x_2\partial_5, \quad X_3=\partial_3-x_1\partial_4-x_2\partial_5, \quad X_4=\partial_4,\, X_5=\partial_5,
$$
see \cite[pages 21-22]{BLD19}.

It can be shown through tedious computations involving the explicit expressions of $X_1,X_2$ above that if a distribution $f$ on $\mathbb F_{23}$ is such that $X_1f=0$, and $X_2^2f=0$, then $f$ is represented by a polynomial in exponential chart (this comes from \cref{thm:Intro2}) and the vector space of such $f$'s is $\mathrm{span}\{1,x_2,x_3,x_4,x_2x_4-x_3^2/2,x_5+x_2x_3/2\}$.

\vspace{0.2cm}
\textbf{SL(2,$\mathbb R$).} Let $\mathrm{SL}(2,\mathbb R)$ be the group of $2\times 2$ real matrices with determinant equal to one. Every element of $\mathrm{SL}(2,\mathbb R)$ in a neighbourhood of the identity can be written as 
$$
\begin{bmatrix}
	x_1 & x_2  \\
	x_3 & \frac{1+x_2x_3}{x_1} \\
\end{bmatrix},
$$
for some $(x_1,x_2,x_3)$ in a neighbourhood of $(1,0,0)$. Thus we can use $(x_1,x_2,x_3)$ as coordinates from an open neighbourhood of $(1,0,0)$ in $\mathbb R^3$ to an open neighbourhood of the identity matrix in $\mathrm{SL}(2,\mathbb R)$. It can be computed, see \cite[Example 7.16]{Tor12}, that, in such a neighbourhood of $(1,0,0)$, the left-invariant vector fields $X_1,X_2,X_3$ such that $(X_i)_{|_{(1,0,0)}}=(\partial_i)_{|_{(1,0,0)}}$ for all $1\leq i\leq 3$, are 
$$
X_1=x_1\partial_1-x_2\partial_2+x_3\partial_3, \quad X_2=x_1\partial_2, \quad X_3=x_2\partial_1+\frac{1+x_2x_3}{x_1}\partial_3.
$$
We have that $[X_2,X_3]=X_1$, and then $\{X_2,X_3\}$ Lie generates the lie algebra of $\mathrm{SL}(2,\mathbb R)$. The coordinate function $x_3$ satisfies $X_2x_3=0$ and $X_3^2x_3=X_3((1+x_2x_3)/x_1)=-x_2(1+x_2x_3)/x_1^2+(1+x_2x_3)/x_1\cdot x_2/x_1=0$. Thus the coordinate function $x_3$ is $\{X_2,X_3\}$-polynomial but $X_1^kx_3=x_3$ for every $k\geq 0$, so that in the previous coordinates $x_3$ is not a polynomial à la Leibman on $\mathrm{SL}(2,\mathbb R)$ even if it is $\{X_1,X_2\}$-polynomial.

\vspace{0.2cm}
\textbf{Orientation-preserving affine functions on $\mathbb R$.}
	This example shows that a $k$-polynomial distribution with respect to a subset $S$ that Lie generates $\mathfrak g$ may not be a polynomial distribution according to \cref{def:PolyGeneral}, and thus it may not be a polynomial à la Leibman, see \cref{prop:Equ1Intro}. Let us consider the Lie group of orientation-preserving affinity of the real line
	\begin{equation}\label{ex:PositiveAffine}
	\mathrm{Aff}^+(\mathbb R):=\{\psi:t\in\mathbb R\mapsto yt+x:x,y\in\mathbb R, y>0\},
	\end{equation}
	endowed with the product
	$$
	(x,y)\cdot(\overline x,\overline y)=(y\overline x+x,y\overline y),
	$$
	that comes from the composition of maps. We identify $\mathrm{Aff}^+(\mathbb R)$ with $\mathbb R\times (0,+\infty)$ by means of the choice of coordinates $(x,y)$. The identity element of the group is $(0,1)$ and the left-invariant vector fields $X,Y$ such that 
	$$
	X_{|_{(0,1)}}=(\partial_x)_{|_{(0,1)}}, \qquad Y_{|_{(0,1)}}=(\partial_y)_{|_{(0,1)}},
	$$
	are
	$$
	X_{|_{(x_0,y_0)}}=y_0(\partial_x)_{|_{(x_0,y_0)}}, \qquad Y_{|_{(x_0,y_0)}}=y_0(\partial_y)_{|_{(x_0,y_0)}}, \qquad \text{for all $(x_0,y_0)\in\mathbb R\times (0,+\infty)$}.
	$$
	We claim that the analytic function $f(x,y):=(x+1)\log y$ on $\mathrm{Aff}^+(\mathbb R)$ is $2$-polynomial with respect to $\{X,Y\}$ but it is not polynomial according to \cref{def:PolyGeneral}, and thus it is not a polynomial à la Leibman, see \cref{prop:Equ1Intro}. Indeed, first $X^2f=(y\partial_x)^2((x+1)\log y)=0$, and $Y^2f=(y\partial_y)^2((x+1)\log y)=0$, and then $f$ is 2-polynomial with respect to $\{X,Y\}$. Second, notice that for every $\alpha\in\mathbb N$ we have, by induction, that $Y^{\alpha}Xf = (y\partial_y)^{\alpha}(y\partial_x)((x+1)\log y)=(y\partial_y)^{\alpha}(y\log y)=y\log y+\alpha y$: thus $f$ cannot be a polynomial according to \cref{def:PolyGeneral}, and then it is not a polynomial à la Leibman, see \cref{prop:Equ1Intro}.
	
	Let us claim moreover that $f$ is not $\mathfrak g$-polynomial, even if it is $S$-polynomial, thus showing that there is no propagation of the property of being $S$-polynomial with a Lie generating $S$ in the non-nilpotent case. Indeed, it can be proved by induction that $(X+Y)^nf=y\log y+ny$ for every $n\geq 2$, and $(X+Y)f=y\log y+x+1$. Thus there does not exist any $n\geq 0$ such that $(X+Y)^nf\equiv 0$. 
	
	Let us further notice that the map $(\alpha,\beta)\to f\circ\exp(\alpha X+\beta Y)$ is not polynomial. Indeed, simple computations lead to show that $\exp(\alpha X+\beta Y)=(\alpha/\beta(e^{\beta}-1),e^{\beta})$ for every $(\alpha,\beta)\in\mathbb R\times(\mathbb R\setminus\{0\})$, while $\exp(\alpha X)=(\alpha,1)$, for every $\alpha\in\mathbb R$. Then $f\circ\exp(\alpha X+\beta Y)=\alpha(e^{\beta}-1)+\beta$, for every $(\alpha,\beta)\in\mathbb R^2$, which is not a polynomial in $(\alpha,\beta)\in\mathbb R^2$.

\end{document}